\documentclass[12pt]{amsart}
\usepackage{amssymb}
\newtheorem{theorem}{Theorem}[section]
\newtheorem{lemma}[theorem]{Lemma}

\theoremstyle{definition}
\newtheorem{definition}[theorem]{Definition}
\newtheorem{example}[theorem]{Example}

\theoremstyle{remark}
\newtheorem{remark}[theorem]{Remark}

\numberwithin{equation}{section}

\newcommand\C{\mathbb{C}}

\newcommand\K{\mathbb{K}}
\newcommand\Z{\mathbb{Z}}

\newcommand\R{\mathbb{R}}
\newcommand\N{\mathbb{N}}
\newcommand\T{\mathbb{T}}
\newcommand\cA{\mathcal{A}}

\newcommand\cE{\mathcal{E}}
\newcommand\cD{\mathcal{D}}
\newcommand\cF{\mathcal{F}}

\newcommand\cM{\mathcal{M}}
\newcommand\cN{\mathcal{N}}

\newcommand\cU{\mathcal{U}}

\newcommand\fA{\mathfrak{A}}

\newcommand\IM{\mathrm{Im}\,}

\newcommand\inpr[2]{\langle{#1,#2}\rangle}
\newcommand{\HS}{\mathrm{H.S.}}
\newcommand{\Ad}{\operatorname{Ad}}
\newcommand{\Aut}{\operatorname{Aut}}
\newcommand{\id}{\operatorname{id}}
\newcommand{\tr}{\operatorname{tr}}
\newcommand{\tK}{\tilde{K}}

\begin{document}
\title{Toeplitz CAR flows and type I factorizations} 

\author{Masaki Izumi}
\address{Department of Mathematics\\ Graduate School of Science\\
Kyoto University\\ Sakyo-ku, Kyoto 606-8502\\ Japan}
\email{izumi@math.kyoto-u.ac.jp}
\thanks{Supported in part by the Grant-in-Aid for Scientific Research (C) 19540214, JSPS}

\author{R. SRINIVASAN}
\address{Chennai Mathematical Institute\\ 
Siruseri 603103, India.}
\email{vasanth@cmi.ac.in}



\begin{abstract} 
Toeplitz CAR flows are a class of $E_0$-semigroups including the first type III example constructed by R. T. Powers. 
We show that the Toeplitz CAR flows contain uncountably many mutually non cocycle conjugate $E_0$-semigroups 
of type III. 
We also generalize the type III criterion for Toeplitz CAR flows employed by Powers 
(and later refined by W. Arveson), and show that Toeplitz CAR flows are always 
either of type I or type III. 
\end{abstract}

\maketitle

\section{Introduction}\label{intro}
The famous E. Wigner's  theorem establishes that any one parameter group of automorphisms $\{\alpha_t: t \in \R\}$ on $B(H)$, 
the algebra of all bounded operators on a separable Hilbert space $H$, is described by a strongly continuous one-parameter unitary group $\{U_t\}$, 
through the relation 
$$\alpha_t(X)=\Ad(U_t)(X)=U_tXU_t^*,~~\forall ~~X \in B(H).$$ 
An analogous statement of Wigner's theorem for an $E_0$-semigroup, a continuous semigroup of unit preserving 
endomorphisms of  $B(H)$, would be that the semigroup is completely determined by the set of 
all intertwining semigroup of isometries. That is the $E_0$-semigroup $\{\alpha_t: t \in (0,\infty)\}$ is completely described, up to cocycle conjugacy, 
by the set of all $C_0$-semigroups of isometries $\{U_t\}$, satisfying 
$$\alpha_t(X)U_t = U_tX, ~~\forall ~~X \in B(H).$$  
A subclass of $E_0$-semigroups, where this analogy is indeed true, are  called as type I $E_0$-semigroups. 
But due to the existence of type II and type III $E_0$-semigroups in abundance, it is well known by now that such an analogy does not hold for 
$E_0$-semigroups in general.

In \cite{PO0}, Powers raised the question whether such an intertwining semigroup of isometries always exists for any given $E_0$-semigroup. 
Later in 1987, he answered this question (see \cite{Po1}) in negative, by constructing an $E_0$-semigroup without any intertwining semigroup of isometries. 
This is the first example of what is called as a type III $E_0$-semigroup. 
For quite some time this was the only known example of a type III $E_0$-semigroup, 
even though it was conjectured that there are uncountably many type III $E_0$-semigroups, which are mutually non cocycle conjugate.  
In 2000, B. Tsirelson  constructed a one-parameter family of nonisomorphic product systems of type III (see \cite{T1}). 
Using previous results of Arveson \cite{Ar}, this leads to the existence of uncountably many $E_0$-semigroups of type III, 
which are mutually non cocycle conjugate.  
Since then there has been a flurry of activity along this direction (see \cite{pdct}, \cite{I} and \cite{IS}). 

In this paper, we turn our attention to the first example of a type III $E_0$-semigroup produced by Powers, 
which can be constructed on the type I factor obtained through the GNS construction 
of the CAR algebra corresponding to a (non-vacuum) quasi-free state. 
Although his purpose in \cite{Po1} is to construct a single type III example, his construction is rather general, 
and it could produce several $E_0$-semigroups, by varying the associated quasi-free states. 
However, it is not at all clear whether they contain more than one cocycle conjugacy classes of type III $E_0$-semigroups. 
As is emphasized in Arveson's book \cite[Chapter 13]{Arv}, 
the 2-point function of Powers' quasi-free state is given by a Toeplitz operator whose symbol is a matrix-valued function 
with a very subtle property. 
Arveson clarified the role of the Toeplitz operator in Powers' construction, and gave the most general form of the symbols for 
which the same construction works. 
We refer to the $E_0$-semigroups obtained in this way as the \textit{Toeplitz CAR-flows}. 
Arveson also made a refinement of a sufficient condition obtained by Powers for the Toeplitz CAR flows to be of type III. 

One of our main purposes in this paper is to show that there exist uncountably many cocycle conjugacy classes of type III Toeplitz CAR flows. 
More precisely, we explicitly give a one parameter family of symbols, including Powers' one, that give rise to mutually non cocycle 
conjugate type III examples. 
We also generalize Powers and Arveson's type III criterion mentioned above, and give a necessary and sufficient condition in full generality, 
which solves Arveson's problem raised in \cite[p.417]{Arv}. 
In particular, our result says that Toeplitz CAR flows are always either of type I or of type III, which is a CAR version of the same result 
obtained in \cite{pdct} (see also \cite{I1} and \cite{IS}) for product systems arising from sum systems, or equivalently, generalized CCR flows. 

As in our previous work \cite{IS}, we employ the local von Neumann algebras of an $E_0$-semigroup as a classification invariant. 
In \cite{IS}, we computed the type of the von Neumann algebras corresponding to bounded open subsets of $(0,\infty)$ 
for a class of generalized CCR flows. 
The key fact in our previous computation is that the von Neumann algebras in question always arise from quasi-free representations of the 
Weyl algebra. 
Since an analogous statement does not seem to be true in the case of Toeplitz CAR flows  
(even if the usual twisting operation in the duality for the CAR algebra is taken into account), we have to take an alternative approach. 
For this reason, we use the notion of a type I factorization, introduced by H. Araki and J. Woods \cite{AW}, consisting of 
the local von Neumann algebras corresponding to a countable partition of a finite interval.  
For each fixed such partition, whether the associated type I factorization is a complete atomic Boolean algebra of type I factors or not 
is a cocycle conjugacy invariant of type III $E_0$-semigroups. 

Part of this work was done when the first named author visited the University of Rome ``Tor Vergata", and 
he would like to thank Roberto Longo and his colleagues for their hospitality.  
The second named author would like to thank the `CMI-TCS Academic Initiative' for the travel support to visit University of Kyoto, during which a part of this work was done.

\section{Preliminaries.}\label{pre}
We use the following notation throughout the paper. 

For a family of von Neumann algebras $\{\cM_{\lambda}\}_{\lambda\in \Lambda}$ acting on the same Hilbert space $H$, 
we denote by $\bigvee_{\lambda\in \Lambda}\cM_\lambda$ the von Neumann algebra generated by their union 
$\bigcup_{\lambda\in \Lambda}\cM_\lambda$. 
We will always denote by $1$ either the identity element in a $C^*$-algebra or the identity operator on a Hilbert space. 
When we need to specify the $C^*$-algebra $\fA$ or the Hilbert space $H$, we use the symbols $1_\fA$ or $1_H$ respectively. 

For a bounded positive operator $A$ on a Hilbert space $H$, we denote by $\tr(A)$ the usual trace of $A$, which could be infinite. 
For $X\in B(H)$, we denote its Hilbert-Schmidt norm by $\|X\|_\HS=\tr(X^*X)^{1/2}$.

For a tempered distribution $f$ on $\R$, we denote by $\hat{f}$ the Fourier transform of $f$ with normalization 
$$\hat{f}(p)=\int_\R f(x)e^{-ipx}dx,\quad f\in L^1(\R).$$ 
For an open set $O\subset \R$, we denote by $\cD(O)$ the set of smooth functions with compact support. 
For a measurable set $E\subset \R$, we denote by $|E|$ and $\chi_E$ its Lebesgue measure and its characteristic function 
respectively.

\subsection{$E_0$-semigroups and product systems}

We briefly recall basics of $E_0$-semigroups and product systems. 
The reader is referred to Arveson's monograph \cite{Arv} for details. 

\begin{definition} 
Let $H$ be a separable Hilbert space. 
A family of unital $*$-endomorphisms $\alpha=\{\alpha_t\}_{t \geq 0}$ of $B(H)$ is an $E_0$-semigroup if 
\begin{itemize}
\item[(i)] The semigroup relation $\alpha_s \circ \alpha_t = \alpha_{s+t}$ holds for $\forall~s,t \in (0,\infty)$ and $\alpha_0=\id$. 
\item[(ii)] The map $t \mapsto \langle \alpha_t(X)\xi,\eta\rangle$ is continuous for every fixed $X \in B(H),~~\xi,\eta \in H$.
\end{itemize}
\end{definition}

For an $E_0$-semigroup $\alpha=\{\alpha_t\}_{t \geq 0}$ and positive $t$, we set 
$$\cE_\alpha(t)=\{T \in B(H); \alpha_t(X) T = TX, ~ \forall ~X \in
B(H)\},$$
which is a Hilbert space with the inner product $\langle T, S \rangle 1_{H}= S^*T$. 
The system of Hilbert spaces $\cE_\alpha=\{\cE_\alpha(t)\}_{t>0}$ satisfies the following axioms of a product system: 

\begin{definition}\label{productsystem}
A product system of Hilbert spaces is a one parameter family of separable complex Hilbert spaces $E=\{E(t)\}_{t>0}$, 
together with unitary operators 
$$U_{s,t} : E(s) \otimes E(t) \rightarrow E(s+t)~ \mbox{for}~ s, t \in (0,\infty),$$
satisfying the following two axioms of associativity and measurability.
\begin{itemize}
\item [(i)] (Associativity) For any $s_1, s_2, s_3 \in (0,\infty)$,
$$U_{s_1, s_2 + s_3}( 1_{E(s_1)} \otimes U_{s_2 ,s_3})=  
U_{s_1+ s_2 , s_3}( U_{s_1 ,s_2} \otimes 1_{E(s_3)}).$$
\item [(ii)] (Measurability) There exists a countable set $E^0$ of
sections $$(0,\infty)\ni t \rightarrow h_t \in E(t)$$ such that $ t  \mapsto 
\langle h_t, h_t^\prime\rangle$ is measurable for any two $h, h^\prime \in E^0$, and the set $\{h_t; h \in E^0\}$ is total 
in $E(t)$, for each $ t \in (0,\infty)$. 
Further it is also assumed that the map $(s,t) \mapsto \langle U_{s,t}(h_s \otimes h_t), h^\prime_{s+t} \rangle$ is measurable 
for any two $h , h^\prime \in E^0$.
\end{itemize}
\end{definition}

Two product systems $(\{E(t)\}, \{U_{s,t}\})$ and $(\{E^\prime(t)\}, \{U^\prime_{s,t}\})$
are said to be isomorphic if there exists a unitary operator $V_t:E(t) \rightarrow E(t)^\prime$, for each
$t \in (0,\infty)$ satisfying $$V_{s+t}U_{s,t}= U_{s,t}^\prime (V_s \otimes V_t).$$

Arveson showed that every product system is isomorphic to a product system arising from an $E_0$-semigroup, and 
that  two $E_0$-semigroups $\alpha$ and $\beta$ are cocycle conjugate if and only if the corresponding product systems 
$\cE_\alpha$ and $\cE_\beta$ are isomorphic. 

For a fixed positive number $a$ and for $0\leq s\leq t\leq a$, we define the local von Neumann algebra $\cA^E_a(s,t)\subset B(E(a))$ 
for the interval $(s,t)$ by 
$$\cA_a^E(s,t)=U_{s,t-s,a-t}(\C1_{E(s)}\otimes B(E(t-s))\otimes \C1_{E(a-t)}){U_{s,t-s,a-t}}^*,$$
where $U_{s,t-s,a-t}=U_{t,a-t}(U_{s,t-s}\otimes 1_{E_{a-t}})=U_{s,a-s}(1_{E_s}\otimes U_{t-s,a-t})$. 
For any open subset $O\subset [0,a]$, we set $\cA^E_a(O)=\bigvee_{I\subset O}\cA^E(I)$, where $I$ runs over all 
intervals contained in $O$. 
When $a=1$, we simply write $\cA^E(s,t)$ for $\cA^E_a(s,t)$. 
When $E=\cE_\alpha$, we often identify $B(E(a))$ with $B(H)\cap\alpha_a(B(H))'$. 
When we need to distinguished them, we denote by $\sigma$ the isomorphism from $B(H)\cap\alpha_a(B(H))'$ onto 
$\cA^E_a(0,a)$ given by the left multiplication. 
By this identification, the inclusion $\cA^E_a(s,t)\subset B(E(a))$ is identified with 
$$\alpha_s(B(H)\cap \alpha_{t-s}(B(H))')\subset B(H)\cap \alpha_a(B(H))'.$$ 

In what follows, we often omit $U_{s,t}$, and simply write $xy$ instead of $U_{s,t}(x\otimes y)$ if there is no possibility 
of confusion. 

\begin{definition}\label{unit}
A unit for a product system $E$ is a non-zero section $$u=\{u_t\in E_t;t > 0 \},$$ such that the map $t \mapsto \langle u_t, h_t\rangle$ is
measurable for any $h \in E^0$ and 
$$u_s u_t= u_{s+t},   \quad \forall s,t \in (0,\infty).$$
\end{definition}

In order to avoid possible confusion, we refer to the condition $\|x\|=1$ for a vector $x\in E(t)$ as ``normalized" instead of ``unit" throughout the paper.   
An intertwining $C_0$-semigroup of isometries of an $E_0$-semigroup $\alpha$ is naturally identified with a normalized unit for $\cE_\alpha$. 

A product system ($E_0$-semigroup) is said to be of type I, if units exist for the product system and 
they generate the product system, i.e. for any fixed $t \in (0,\infty)$, the set $$\{u^1_{t_1}u^2_{t_2}  
\cdots u^n_{t_n}; \sum_{i=1}^n t_i = t, u^i \in \cU_E\},$$ 
is a total set in $E_t$, where $\cU_E$ is the set of all units. 
It is of type II if units exist but they do not generate the product system. 
An $E_0$-semigroup is said to be spatial if it is either of type I or type II. 
We say a product system to be of type III, or unitless, if there does not exist any unit.

Type I product systems are further classified into type I$_n$, $n=1,2,\cdots,\infty$, according to 
their indices $n$. 
There exists only one isomorphism class of type $I_n$ product systems.

We recall V. Liebscher's useful criterion  \cite[Corollary 7.7]{L} for isomorphic product systems in terms of 
the local von Neumann algebras. 

\begin{theorem}[Liebscher]\label{isomorphic} Let $E$ and $F$ be product systems. 
If there is an isomorphism $\rho$ from $B(E(1))$ onto $B(F(1))$ such that 
$\rho(\cA^E(0,t))=\cA^F(0,t)$ for $t$ in a dense subset of $(0,1)$, then 
$E$ and $F$ are isomorphic. 
\end{theorem} 

\subsection{Type I factorizations}\label{invariant for type III}
In this subsection, we introduce a new classification invariant for type III product systems using the notion of type I factorizations 
introduced by Araki and Woods \cite{AW}. 
Throughout this subsection, every index set is assumed to be countable, and every Hilbert space is assumed to be separable. 

\begin{definition} 
Let $H$ be a Hilbert space. 
We say that a family of type I subfactors $\{\cM_{\lambda}\}_{\lambda\in \Lambda}$ of $B(H)$ is a \textit{type I factorization} of $B(H)$ if 
\begin{itemize}
\item [(i)] $\cM_\lambda\subset \cM_\mu'$ for any $\lambda,\mu\in \Lambda$ with $\lambda\neq \mu$, 
\item [(ii)] $B(H)=\bigvee_{\lambda\in \Lambda}\cM_\lambda$.
\end{itemize}
We say that a type I factorization $\{\cM_{\lambda}\}_{\lambda\in \Lambda}$ is a \textit{complete atomic Boolean algebra of type I factors} 
(abbreviated as \textit{CABATIF}) if for any subset $\Gamma\subset \Lambda$, the von Neumann algebra $\bigvee_{\lambda\in \Gamma}\cM_{\lambda}$ is a type I factor. 
\end{definition}

Two type I factorizations $\{\cM_{\lambda}\}_{\lambda\in \Lambda}$ of $B(H)$ and $\{\cN_{\mu}\}_{\mu\in \Lambda'}$ of $B(H')$ 
are said to be unitarily equivalent if there exist a unitary $U$ from $H$ onto $H'$ and a bijection $\sigma:\Lambda\rightarrow \Lambda'$ 
such that $U\cM_\lambda U^*=\cN_{\sigma(\lambda)}$. 

\begin{example}\label{invariant for E} Let $E$ be a product system, and let $\{a_n\}_{n=0}^\infty$ be 
a strictly increasing sequence of non-negative numbers starting 
from $0$ and converging to $a<\infty$. 
Then $\{\cA_a^E(a_n,a_{n+1})\}_{n=0}^\infty$ is a type I factorization of $B(E(a))$ because 
$$B(E(a))=\bigvee_{0<t<a}\cA^E_a(0,t)$$ 
holds (see \cite[Proposition 4.2.1]{Arv}). 
For a fixed sequence as above, the unitary equivalence class of the type I factorization $\{\cA_a^E(a_n,a_{n+1})\}_{n=0}^\infty$ is an isomorphism 
invariant of the product system $E$.  
In particular, whether it is a CABATIF  
or not will be used to distinguish concrete type III examples  in Section \ref{examples}. 
As we will see now, this invariant may be useful only in the type III case. 
\end{example}

When $\{\cM_\lambda\}_{\lambda\in \Lambda}$ is a type I factorization of $B(H)$, we say that a non-zero vector $\xi$ is \textit{factorizable} 
if for any $\lambda$, there exists a minimal projection $p_\lambda$ of $\cM_\lambda$ such that $p_\lambda\xi=\xi$. 

Araki and Woods characterized a CABATIF as a type I factorization with a factorizable vector. 
Since we need a more precise statement, we briefly recall basics of the incomplete tensor product space 
(abbreviated as ITPS) now. 

Let $\{(H_\lambda,\xi_\lambda)\}_{\lambda\in \Lambda}$ be a family of Hilbert spaces $H_\lambda$ with normalized vectors 
$\xi_\lambda\in H_\lambda$. 
Let $\cF(\Lambda)$ be the set of all finite subsets of $\Lambda$, which is a directed set with respect to the inclusion relation. 
For $F_1,F_2\in \cF(\Lambda)$ with $F_1\subset F_2$, we introduce an isometric embedding $V_{F_1,F_2}$ from $\bigotimes_{\lambda\in F_1}H_\lambda$ into 
$\bigotimes_{\lambda\in F_2}H_\lambda$ by 
$$V_{F_1,F_2}\eta= \eta\otimes (\bigotimes_{\mu\in F_2\setminus F_1}\xi_{\mu}),\quad \eta\in \bigotimes_{\lambda\in F_1}H_\lambda. $$
Then the ITPS 
$$H=\bigotimes_{\lambda\in \Lambda}{}^{(\otimes \xi_{\lambda})} H_\lambda$$ 
of the Hilbert spaces $\{H_\lambda\}_{\lambda\in \Lambda}$, with respect to the reference vectors 
$\{ \xi_\lambda \}_{\lambda \in \Lambda}$, is the completion of the direct limit of the directed family 
$\{\bigotimes _{\lambda\in F}H_\lambda\}_{F\in \cF(\Lambda)}$. 
When there is no possibility of confusion, we omit the superscript $(\otimes \xi_{\lambda})$ for simplicity. 
We denote by $V_{F,\infty}$ the canonical embedding of $\bigotimes_{\lambda\in F}H_\lambda$ into $H$. 

The product vector $\xi=\bigotimes_{\lambda\in \Lambda} \xi_\lambda\in H$ is understood as 
$V_{F,\infty}\bigotimes_{\lambda\in F}\xi_\lambda$, which does not depend on $F\in \cF(\Lambda)$. 
More generally, if $\{\eta_\lambda\}_{\lambda\in\Lambda}$, $\eta_\lambda\in H_\lambda$, is a family of vectors such that 
$0<\prod_{\lambda\in \Lambda}\|\eta_\lambda\|<\infty$, and 
$$\sum_{\lambda\in \Lambda}|\inpr{\eta_\lambda}{\xi_\lambda}-1|<\infty,$$
then the net $\{V_{F,\infty}\bigotimes_{\lambda\in F} \eta_\lambda\}_{F\in \cF(\Lambda)}$ converges in $H$. 
The product vector $\eta=\bigotimes_{\lambda\in \Lambda} \eta_\lambda$ is defined as its limit. 
Two product vectors $\eta=\bigotimes_{\lambda\in \Lambda}\eta_\lambda$ and $\zeta=\bigotimes_{\lambda\in \Lambda}\zeta_\lambda$
satisfy
$$\inpr{\eta}{\zeta}=\prod_{\lambda\in \Lambda}\inpr{\eta_\lambda}{\zeta_\lambda}.$$

For a subset $\Lambda_1\subset \Lambda$, we often identify $\bigotimes_{\lambda\in \Lambda}H_\lambda$ with 
$$(\bigotimes_{\lambda\in \Lambda_1}H_\lambda)\otimes(\bigotimes_{\mu\in \Lambda\setminus\Lambda_1}H_\mu)$$ 
in a canonical way. 
When $\Lambda_1$ consists of only one point $\lambda$, we set 
$$\cM_\lambda:=B(H_\lambda)\otimes \C1_{\bigotimes_{\mu\neq \lambda}H_\mu}\subset B(H).$$ 
Then $\{\cM_\lambda\}_{\lambda\in \Lambda}$ is a CABATIF. 
Any type I factorization unitarily equivalent to this $\{\cM_\lambda\}_{\lambda\in \Lambda}$ is said to be a \textit{tensor product factorization}. 
Note that there is only one tensor product factorization, up to unitary equivalence, with each constituent type I factor infinite dimensional. 

One can find the following theorem in \cite[Lemma 4.3, Theorem 4.1]{AW}.  

\begin{theorem}[Araki--Woods]\label{ArakiWoods} 
A type I factorization is a CABATIF if and only if it has a factorizable vector. 
When this condition holds, then it is a tensor product factorization. 
\end{theorem} 

When a product system $E$ has a unit, then it gives a factorizable vector of the type I factorization 
$\{\cA_a^E(a_n,a_{n+1})\}_{n=0}^\infty$ in Example \ref{invariant for E}, which is necessarily a CABATIF 
thanks to Theorem \ref{ArakiWoods}. 

We use the following lemma in Section \ref{factorization}. 

\begin{lemma}\label{fixed product vector} Let $H=\bigotimes_{\lambda\in \Lambda}H_\lambda$ be the ITPS of Hilbert spaces 
$\{H_\lambda\}_{\lambda\in \Lambda}$ with respect to reference vectors $\{\xi_\lambda\}_{\lambda\in \Lambda}$, and let 
$$\rho_\lambda:B(H_\lambda)\ni X\mapsto X\otimes 1_{\bigotimes_{\mu\neq \lambda}H_\mu}\in \cM_\lambda$$ 
be the canonical isomorphism. 
Let $R\in B(H)$ be a self-adjoint unitary such that $R\cM_\lambda R^*=\cM_\lambda$ for all $\lambda\in \Lambda$. 
Then there exist self-adjoint unitaries $R_\lambda\in B(H_\lambda)$ and a product vector 
$\eta=\bigotimes_{\lambda\in \Lambda}\eta_\lambda$ such that 
\begin{itemize}
\item [(i)]  for $\forall \lambda\in \Lambda$ and $\forall X\in \cM_\lambda$,
$$\rho_\lambda(R_\lambda) X \rho_\lambda(R_\lambda^*)=RXR^*,$$
\item [(ii)] $R_\lambda \eta_\lambda=\eta_\lambda$ for $\forall \lambda\in \Lambda$, 
\item [(iii)] either $R\eta=\eta$ or $R\eta=-\eta$. 
\end{itemize}
\end{lemma}

\begin{proof} Since the restriction of $\Ad R$ to $\cM_\lambda$ is an automorphism of period two and $\cM_\lambda$ is a type I factor, 
there exist self-adjoint unitaries $R_\lambda\in B(H_\lambda)$ such that 
$$\rho_\lambda(R_\lambda)X\rho_{\lambda}(R_\lambda^*)=RXR^*,\quad \forall X\in \cM_\lambda.$$ 
By replacing $R_\lambda$ with $-R_\lambda$ if necessary, we may assume $\inpr{R_\lambda \xi_\lambda}{\xi_\lambda}\geq 0$ for 
$\forall \lambda\in \Lambda$.  
Let $\xi=\bigotimes_{\lambda\in \Lambda}\xi_\lambda$, and let $p_\lambda\in \cM_\lambda$ be the minimal projection satisfying 
$p_\lambda \xi=\xi$. 
Then $q_\lambda=Rp_\lambda R^*$ is a minimal projection of $\cM_\lambda$ satisfying 
$q_\lambda R\xi=R\xi$, and so $R\xi$ is a factorizable vector. 
The proof of \cite[Lemma 3.2]{AW} shows that there exist a complex number $c$ of modulus 1 and  normalized vectors $\zeta_\lambda\in H_\lambda$ 
such that $R\xi=c\bigotimes_{\lambda\in \Lambda} \zeta_\lambda$ and $\inpr{\zeta_\lambda}{\xi_\lambda}\geq 0$ for $\forall \lambda\in \Lambda$. 
Since $q_\lambda=\rho_\lambda(R_\lambda) p_\lambda \rho_\lambda(R_\lambda^*)$, the normalized vector $\zeta_\lambda$ is a scalar multiple of 
$R_\lambda\xi_\lambda$. 
Let 
$$\Lambda_0=\{\lambda\in\Lambda;\; \inpr{\zeta_\lambda}{\xi_\lambda}=0 \},$$
and $\Lambda_1=\Lambda\setminus \Lambda_0$. 
Then $\Lambda_0$ is a finite set. 
The two conditions $\inpr{R_\lambda \xi_\lambda}{\xi_\lambda}\geq 0$ and $\inpr{\zeta_\lambda}{\xi_\lambda}\geq 0$ 
imply that for any $\lambda\in \Lambda_1$, we actually have 
$R_\lambda\xi_\lambda=\zeta_\lambda$. 
Let $Q_\lambda$ be the spectral projection of $R_\lambda$ corresponding to eigenvalue 1. 
Then since $R_\lambda=2Q_\lambda-1$, we have $\inpr{\zeta_\lambda}{\xi_\lambda}=2\inpr{Q_\lambda\xi_\lambda}{\xi_\lambda}-1$. 

For $\lambda\in \Lambda_0$, by replacing $R_\lambda$ with $-R_\lambda$ if necessary, we can find a normalized vector $\eta_\lambda\in H_\lambda$ 
satisfying $R_\lambda\eta_\lambda=\eta_\lambda$. 
For $\lambda\in \Lambda_1$, we set $\eta_\lambda=Q_\lambda\xi_\lambda$. 
Then 
$$\|\eta_\lambda\|^2=\inpr{\eta_\lambda}{\xi_\lambda}=\frac{1+\inpr{\zeta_\lambda}{\xi_\lambda}}{2}.$$
This shows that the product vector $\bigotimes_{\lambda\in \Lambda}\eta_\lambda\in H$ exists and $R_\lambda \eta_\lambda=\eta_\lambda$. 

It only remains to show (iii). 
Let $e_\lambda\in B(H_\lambda)$ be the projection onto $\C\eta_\lambda$. 
Then the proof of \cite[Lemma 3.2]{AW} shows that the net $\{\prod_{\lambda\in F}\rho_{\lambda}(e_\lambda)\}_{F\in \cF(\Lambda)}$ 
strongly converges to the projection $e\in B(H)$ onto $\C\eta$. 
Since $ReR^*=e$ and $R$ is a self-adjoint unitary, we get either $R\eta=\eta$ or $R\eta=-\eta$. 
\end{proof}

\subsection{Quasi-free representations of the CAR algebra}
We recall some of the well-known results about quasi-free representations of the algebra of canonical anticommutation relations 
(called as CAR algebra). 

Let $K$ be a complex Hilbert space.  
We denote by $\fA(K)$ the CAR algebra over $K$, which is the universal $C^*$-algebra generated by 
$\{a(x); x \in K\}$, determined by the linear map $x \mapsto a(x)$ satisfying the CAR relations:  
\begin{eqnarray*}
a(x)a(y) +a(y)a(x)&  = & 0, \\
a(x)a(y)^* +a(y)^*a(x) & = & \langle x,y\rangle1, 
\end{eqnarray*} for all $x,y \in K$. 
Since $\fA(K)$ is known to be simple, any set of operators satisfying the CAR relations generates a $C^*$-algebra 
canonically isomorphic to $\fA(K)$. 

For any state $\varphi$ of $\fA(K)$, there exists a unique positive contraction $A\in B(K)$ satisfying 
$\varphi(a(f)a(g)^*)=\inpr{Af}{g}$ for $\forall f,g\in K$. 
We call $A$ the covariance operator (or 2-point function) of $\varphi$. 

A \textit{quasi-free state} $\omega_A$ on $\fA(K)$, associated with a positive contraction $A \in B(K)$, 
is the state whose $(n,m)$-point functions are determined by its 2-point function as  
$$\omega_A(a(x_n) \cdots a(x_1)a(y_1)^* \cdots a(y_m)^* ) = \delta_{n,m} \det (\langle Ax_i, y_j\rangle) ,$$ 
where $\det(\cdot)$ denotes the determinant of a matrix.  
Given a positive contraction, it is a fact that such a state always exists and is uniquely determined by the above relation. 
This is usually called as the gauge invariant quasi-free state (or generalized free state). 
Since we will be dealing only with gauge invariant quasi-free states, we just call them as quasi-free states. 

We denote by $(H_A, \pi_A, \Omega_A)$ the GNS triple associated with 
a quasi-free state $\omega_A$ on $\fA(K)$, and set $\cM_A:=\pi_A(\fA(K))''$. 
We call $\pi_A$ the \textit{quasi-free representation} associated with $A$.

Recall that two representations  $\pi_1$ and $\pi_2$ of a $C^*$-algebra $\fA$ are said to be quasi-equivalent 
if there is a $*$-isomorphism of von Neumann algebras 
$$\theta:\pi_1(\fA)^{\prime\prime} \longmapsto \pi_2(\fA)^{\prime \prime}$$ satisfying 
$\theta(\pi_1(X))=\pi_2(X)$ for all $X \in \fA$. 
Two states are said to be quasi-equivalent if their GNS representations are quasi-equivalent. 

We now summarize standard results on quasi-free states. 
For the proofs, the reader is referred to \cite[Chapter13]{Arv}, \cite[Section II]{Po1}, and references therein. 

\begin{theorem}\label{qfstate} Let $K$ be a Hilbert space, let $P\in B(K)$ be a projection, and let $A,B\in B(K)$ be positive contractions.  
\begin{itemize}
\item [(i)] Every quasi-free state $\omega_A$ of the CAR algebra $\fA(K)$ is a factor state, that is, 
the von Neumann algebra $\cM_A$ is a factor. 
\item [(ii)] The restriction of the GNS representation $\pi_A$ to $\fA(PK)$ is quasi-equivalent to the GNS representation 
$\pi_{PAP}$ of $\fA(PK)$, where $PAP$ is regarded as a positive contraction of $PK$. 
\item [(iii)] The quasi-free state $\omega_A$ is of type I if and only if $\tr(A -A^2) < \infty.$  
\item [(iv)] The two quasi-free states $\omega_A$ and $\omega_B$ are quasi-equivalent if and only if both operators 
$A^{1/2} - B^{1/2}$ and $(1-A)^{1/2} -(1-B)^{1/2}$ are Hilbert-Schmidt. 
\item [(v)] The two quasi-free states $\omega_A$ and $\omega_P$ are quasi-equivalent if and only if  
$$\tr\big(P(1-A)P+(1-P)A(1-P)\big)<\infty.$$   
\end{itemize}
\end{theorem}

We frequently use the following criterion, which is more or less (v) above.   

\begin{lemma}\label{typeI} Let $A,B\in B(K)$ be positive contractions. 
We assume that $\omega_B$ is a type I state. 
Then the two quasi-free states $\omega_A$ and $\omega_B$ are quasi-equivalent if and only if 
$$\tr\big(B(1-A)B+(1-B)A(1-B)\big)<\infty.$$  
\end{lemma}

\begin{proof} Let $P$ be the spectral projection of $B$ corresponding to the interval $[1/2,1]$. 
Since $\omega_B$ is a type I state, Theorem \ref{qfstate},(iii),(iv) imply that 
$P-B$ is a trace class operator, and $\omega_P$ and $\omega_B$ are quasi-equivalent. 
Thus $\omega_A$ and $\omega_B$ are quasi-equivalent if and only if $\omega_A$ and $\omega_P$ are quasi-equivalent, 
which is further equivalent to 
$$\tr\big(P(1-A)P+(1-P)A(1-P)\big)<\infty,$$  
thanks to Theorem \ref{qfstate},(v).  
Now the statement follows from the fact that $P-B$ is a trace class operator.
\end{proof}

Let $\gamma$ be the period two automorphism of $\fA(K)$ determined by 
$\gamma(a(f))=-a(f)$ for $\forall f\in K$. 
Since any quasi-free state $\omega_A$ is invariant under $\gamma$, the automorphism $\gamma$ extends to a period 
two automorphism $\overline{\gamma}$ of the von Neumann algebra $\cM_A$. 
For a $\Z/2\Z$-grading of $\fA(K)$ (respectively $\cM_A$), we always refer to the one coming from $\gamma$ 
(respectively $\overline{\gamma}$). 
When there is no possibility of confusion, we abuse the notation and use the same symbol $\gamma$ for $\overline{\gamma}$. 

Let $\omega_A$ be a type I state. 
Then since every automorphism of a type I factor is inner, there exists a self-adjoint unitary $R^A\in \pi_A(\fA)''$ 
satisfying $\Ad R^A(X)=\gamma(X)$ for all $X\in \cM_A$. 
The operator $R^A$ is uniquely determined up to a multiple of $-1$. 
In the same way, for every closed subspace $L\subset K$ such that the restriction of $\pi_A$ to $\fA(L)$ is of type I, 
there exists a self-adjoint unitary $R^A_L\in \pi_A(\fA(L))''$ satisfying $\Ad R^A_L(X)=\gamma(X)$ 
for all $ X\in \pi_A(\fA(L))''$. 
For each $L$, we fix such $R^A_L$, which itself is an even operator with respect to $\gamma$. 
When $L_1$ and $L_2$ are mutually orthogonal closed subspaces of $K$ satisfying the above condition, then 
we have 
$$R^A_{L_1\oplus L_2}=\epsilon_{L_1,L_2}R^A_{L_1}R^A_{L_2}=\epsilon_{L_1,L_2}R^A_{L_2}R^A_{L_1},$$
where $\epsilon_{L_1,L_2}\in \{1,-1\}$. 

When $\omega_A$ is of type I, the family of operators $\{i\pi_A(a(f))R^A;\;f\in K \}$ also satisfies the CAR relation. 
We denote by $\pi^t_A$ the representation of $\fA(K)$ determined by $\pi^t_A(a(f))=i\pi_A(a(f))R^A$ for all $f\in K$, 
and call it the twisted representation associated with $\omega_A$. 
Note that the two representations $\pi_A$ and $\pi_A^t$ coincides on the even part of $\fA(K)$. 

\begin{lemma}\label{commutant} Let $\omega_A$ be a type I quasi-free state of $\fA(K)$. 
\begin{itemize}
\item [(i)] For any subspace $L\subset K$,  
$$\cM_A\cap \pi_A(\fA(L))'=\pi_A^t(\fA(L^\perp))''.$$ 
\item [(ii)] Let $U=\frac{1}{\sqrt{2}}(1-iR^A)\in \cM_A$. 
Then $$U\pi_A(X)U^*=\pi_A^t(X)$$  holds for all $X\in \fA(K)$. 
\end{itemize}
\end{lemma}

\begin{proof} (i) Let $Q$ be the spectral projection of $A$ corresponding to the interval $[1/2,1]$. 
Then $\pi_A$ and $\pi_Q$ are quasi-equivalent, and we may assume that $A$ is a projection for the proof by 
replacing $A$ with $Q$ if necessary. 
Now the statement follows from the twisted duality theorem \cite[Theorem 2.4]{F}. 

(ii) This follows from a direct computation (or \cite[Proposition 2.3]{F}).  
\end{proof}

As in \cite{Po1}, we also need to use a few facts about general factor states of $\fA(K)$. 

\begin{lemma}\label{covariance operator} Let $A$ be the covariance operator of a state $\varphi$ of $\fA(K)$. 
Then, 
\begin{itemize}
\item [(i)] If $A$ is a projection, then $\varphi$ is the pure state $\omega_A$. 
\item [(ii)] If $\varphi$ is quasi-equivalent to a quasi-free state $\omega_B$, then $A-B$ is compact. 
\end{itemize}
\end{lemma}

\begin{proof} (i) See, for example, \cite[Lemma 4.3]{Ara}. 

(ii) The statement follows from \cite[Theorem 13.1.3]{Arv}. 
\end{proof}

\subsection{Toeplitz CAR flows}
Let $V$ be an isometry of a Hilbert space $K$. 
Then we have an endomorphism $\rho$ of $\fA(K)$ determined by $\rho(a(f))=a(Vf)$ for all $ f\in K$. 
For a positive contraction $A$, the composition $\pi_A\circ \rho$ gives a representation of $\fA(K)$, 
which is quasi-equivalent to $\pi_{V^*AV}$ thanks to Theorem \ref{qfstate},(ii). 
Thus if both $A^{1/2}-(V^*AV)^{1/2}$ and $(1-A)^{1/2}-(1-V^*AV)^{1/2}$ are Hilbert-Schmidt operators, 
then $\rho$ extends to an endomorphism of the von Neumann algebra $\cM_A$. 
In particular, if $A$ satisfies $\tr(A-A^2)<\infty$ and $\{V_t\}_{t\geq 0}$ is a strongly continuous semigroup of isometries on 
$K$ satisfying the above condition for $V_t$ in place of $V$, then we get an $E_0$-semigroup. 

In what follows, we assume $K=L^2((0,\infty),\C^N)$, and that $\{S_t\}_{t\geq 0}$ is the shift semigroup 
$$S_tf(x)=\left\{
\begin{array}{ll}
0 , &\quad 0<x\leq t, \\
f(x-t) , &\quad t<x.
\end{array}
\right.
$$ 
In his attempt to clarify Powers' construction \cite{Po1} of the first example of a type III $E_0$-semigroup, 
Arveson \cite[Section 13.3]{Arv} determined the most general form of a positive contraction $A\in B(K)$ satisfying $\tr(A-A^2)<\infty$ and 
$S_t^*AS_t=A$ for all $t$, which we state now. 

We regard $K$ as a closed subspace of $\tK:=L^2(\R,\C^N)$, and we denote by $P_+$ the projection from $\tK$ onto $K$. 
We often identify $B(K)$ with $P_{+}B(\tK)P_+$. 

We denote by $M_N(\C)$ the $N$ by $N$ matrix algebra. 
For $\Phi\in L^\infty(\R)\otimes M_N(\C)$, we define the corresponding Fourier multiplier $C_\Phi\in B(\tK)$ by 
$$\hat{(C_\Phi f)}(p)=\Phi(p)\hat{f}(p).$$ 
Then the Toeplitz operator $T_\Phi\in B(K)$ and the Hankel operator $H_\Phi\in B(K,K^\perp)$ with the symbol $\Phi$ are defined by 
$$T_\Phi f=P_+C_\Phi f,\quad f\in K,$$
$$H_\Phi f=(1_{\tK}-P_+)C_\Phi f,\quad f\in K.$$

\begin{theorem}[Arveson]\label{symbol} Let $K=L^2((0,\infty),\C^N)$. 
A positive contraction $A\in B(K)$ satisfies $\tr(A-A^2)<\infty$ and $S_t^*AS_t=A$ if and only if 
there exists a projection $\Phi\in L^\infty(\R)\otimes M_N(\C)$ satisfying the following two conditions: 
\begin{itemize}
\item [(i)] $A=T_\Phi$,
\item [(ii)] the Hankel operator $H_\Phi$ is Hilbert-Schmidt.  
\end{itemize}
\end{theorem}

We call the symbol $\Phi$ satisfying the condition of Theorem \ref{symbol} as \textit{admissible}.

Let $\Phi\in L^\infty(\R)\otimes M_N(\C)$ be a projection. 
Arveson briefly mentioned in \cite[p.401]{Arv}, without giving a proof, that the condition (ii) of Theorem \ref{symbol} holds 
if and only if the Fourier transform $\hat{\Phi}(x)$ (in distribution sense) restricted to $\R\setminus \{0\}$ is locally 
square integrable and  
$$\sup_{\delta>0}\int_{|x|>\delta}|x|\tr(|\hat{\Phi}(x)|^2)dx<\infty. $$
He also observed that any admissible symbol is necessarily quasi-continuous, 
though he used only the fact that $H_\Phi$ is a compact operator.   
Now, first we figure out the most suitable function space for the admissible symbols without using the Fourier transform, 
and then we give a proof to the above characterization in terms of the Fourier transform. 
We will see similarity between admissible symbols and 
logarithm of spectral density functions of off-white noises discussed in \cite{T2}.  

We denote by $\T$  the unit circle in $\C$. 
Let $U$ be the unitary from $L^2(\R)$ onto $L^2(\T,\frac{dt}{2\pi})$ induced by 
the change of variables 
$$e^{it}=-\frac{p+i}{p-i},$$
(since the Fourier transform $\hat{f}(p)$ of $f\in K$ has analytic continuation to the \textit{lower} half-plane, 
we need a conformal transformation between the unit disk and the lower half-plane). 
Let $F$ be the unitaries associated with the Fourier transform. 
Then the Hankel operator $H_\Phi$ is transformed to the Hankel operator $H_\phi$ for $\T$ by $UF$, where 
$\Phi$ and $\phi$ are related by $\phi(e^{it})=\Phi(p)$ (see for example, \cite[Section 3]{T2}). 
Let $\phi_{ij}(p)$ be the matrix element of $\phi(p)$. 
Since $\phi(e^{it})$ is a projection,  the Hankel operator $H_\phi$ 
is Hilbert-Schmidt if and only if $H_{\phi_{ij}}$ and $H_{\overline{\phi_{ij}}}$ are Hilbert-Schmidt for all $i\leq j$. 

It is easy to see that the Hankel operators $H_h$ and $H_{\overline{h}}$ for $h\in L^\infty(\T)$ are Hilbert-Schmidt 
if and only if $h$ is in the Sobolev space $W^{1/2}_2(\T)$, that is 
$$\sum_{n\in \Z}|n| |\hat{h}(n)|^2<\infty,$$
where $\hat{h}(n)$ is the Fourier coefficient 
$$\hat{h}(n)=\frac{1}{2\pi}\int_0^{2\pi}h(e^{it})e^{-int}dt.$$
This is further equivalent to the condition that $h$ belongs to the Besov space $B_{2,2}^{1/2}(\T)$ 
because 
\begin{eqnarray*}
\int_0^{2\pi}\int_0^{2\pi}\frac{|h(e^{is})-h(e^{it})|^2}{|e^{is}-e^{it}|^2}dsdt
&=&\int_0^{2\pi}\int_0^{2\pi}\frac{|h(e^{i(s+t)})-h(e^{it})|^2}{|e^{is}-1|^2}dtds\\
 &=&2\pi\int_0^{2\pi}\sum_{n\in \Z}\frac{|(e^{ins}-1)\hat{h}(n)|^2}{|e^{is}-1|^2}ds\\
 &=&4\pi^2\sum_{n\in \Z}|n| |\hat{h}(n)|^2.
\end{eqnarray*}
As was done in \cite[Section 3]{T2}, we can translate this condition back into that for functions on $\R$. 
Now we see that the Hankel operator $H_\Phi$ with a projection $\Phi\in L^\infty(\R)\otimes M_N(\C)$ is Hilbert-Schmidt 
if and only if 
$$\int_{\R^2}\frac{\tr(|\Phi(p)-\Phi(q)|^2)}{|p-q|^2}dpdq<\infty. $$

Although the following lemma may be found in the literature of Besov spaces, for the reader's convenience, 
we give a proof to the first part.   
(i) and (ii) are essentially due to Tsirelson \cite[Proposition 3,6]{T2}. 

\begin{lemma}\label{Besov} Let $\psi(p)$ be a measurable function on $\R$ giving a tempered distribution, 
and let $0<\mu\leq 1$. 
Then the following two conditions are equivalent: 
\begin{itemize}
\item [$(1)$] The function $\psi$ satisfies 
$$\int_{\R^2}\frac{|\psi(p)-\psi(q)|^2}{|p-q|^{1+\mu}}dpdq<\infty.$$
\item [$(2)$] There exists a measurable function $\hat{\psi}_0(x)$ on $\R$ such that 
$$\int_\R |x|^\mu|\hat{\psi}_0(x)|^2dx<\infty,$$
and $x\hat{\psi}(x)=x\hat{\psi}_0(x)$ as distributions. 
\end{itemize}
Moreover, 
\begin{itemize}
\item [(i)] If $\psi$ satisfies the conditions $(1)$,$(2)$, then  
$$\int_\R |\psi(2p)-\psi(p)|^2\frac{dp}{|p|^\mu}<\infty.$$
\item [(ii)] 
If $\psi$ is an even differentiable function satisfying 
$$\int_0^\infty|\psi'(p)|^2 |p|^{2-\mu}dp<\infty,$$
then $\psi$ satisfies the conditions $(1)$,$(2)$. 
\end{itemize}
\end{lemma}

\begin{proof} Assume that (1) holds. 
Since the condition (1) is written as 
$$\int_{\R^2}\frac{|\psi(p+q)-\psi(q)|^2}{|p|^{1+\mu}}dqdp<\infty,$$
the function $q\mapsto \psi(p+q)-\psi(q)$ is square integrable for almost all $p\in \R$, and so is 
the distribution $(e^{ipx}-1)\hat{\psi}(x)$ by the Plancherel theorem. 
This shows that the restriction of $\hat{\psi}$ to $\cD(\R\setminus \{0\})$ is given by a locally 
square integrable function on $\R\setminus\{0\}$,  say $\hat{\psi}_0(x)$, and that for almost all $p\in \R$, 
the equation 
\begin{equation}\label{distribution}
(e^{ipx}-1)\hat{\psi}(x)=(e^{ipx}-1)\hat{\psi}_0(x) 
\end{equation}
holds as distributions in the variable $x$. 
In the above, we regards $\hat{\psi}_0(x)$ as a measurable function on $\R$ by setting $\hat{\psi}_0(0)=0$.  
Now the Plancherel formula implies 
\begin{eqnarray*}
\int_\R\int_\R\frac{|\psi(p+q)-\psi(q)|^2}{|p|^{1+\mu}}dqdp
&=&\frac{1}{2\pi} \int_\R\int_\R\frac{|(e^{ipx}-1)\hat{\psi}_0(x)|^2}{|p|^{1+\mu}}dxdp\\
&=& \frac{2}{\pi} \int_\R\int_\R\frac{\sin^2\frac{px}{2}|\hat{\psi}_0(x)|^2}{|p|^{1+\mu}}dpdx \\
&=&\frac{2^{1-\mu}}{\pi}\int_\R\frac{\sin^2r}{|r|^{1+\mu}}dr\int_\R|x|^\mu|\hat{\psi}_0(x)|^2dx. 
\end{eqnarray*}
This implies the convergence of the integral in (2), which shows that $x\hat{\psi}_0(x)$ is a tempered distribution. 
Since the support of $x\hat{\psi}(x)-x\hat{\psi}_0(x)$ is contained in $\{0\}$, 
we have   
$$x\hat{\psi}(x)-x\hat{\psi}_0(x)=\sum_{k=0}^nc_k\delta_0^{(k)}(x),$$
where $c_k\in \C$ and $\delta_0$ is the Dirac mass at 0. 
We choose $p\neq 0$ such that (\ref{distribution}) holds, and set 
$$h(x)=\left\{
\begin{array}{ll}
\frac{e^{ipx}-1}{x} , &\quad x\neq 0, \\
ip , &\quad x=0
\end{array}
\right..
$$
Then 
$$0=(e^{ipx}-1)(\hat{\psi}(x)-\hat{\psi}_0(x))=h(x)(x\hat{\psi}(x)-x\hat{\psi}_0(x))=\sum_{k=0}^nc_kh(x)\delta_0^{(k)}(x).$$
It is routine work to show $c_k=0$ for all $k$ from this and  $h(0)\neq 0$, and we get (2). 

By tracing back the same computation as above, we can also show the implication from (2) to (1). 

(i) and (ii) are essentially \cite[Proposition 3,6]{T2}. 
\end{proof}

Summarizing our argument so far, we get 

\begin{theorem}\label{Sobolev} Let $\Phi\in L^\infty(\R)\otimes M_N(\C)$ be a projection. 
Then the following three conditions are equivalent: 
\begin{itemize}
\item [(1)] The symbol $\Phi$ is admissible. \medskip
\item [(2)] 
$$\int_{\R^2}\frac{\tr(|\Phi(p)-\Phi(q)|^2)}{|p-q|^2}dpdq<\infty.$$
\item [(3)] There exists a $M_N(\C)$-valued measurable function $\hat{\Phi}_0(x)$ on $\R$ such that 
$$\int_\R|x|\tr(|\hat{\Phi}_0(x)|^2)dx<\infty,$$
and $x\hat{\Phi}(x)=x\hat{\Phi}_0(x)$ as $M_N(\C)$-valued distributions. 
\end{itemize}
Moreover,  
\begin{itemize}
\item [(i)] If $\Phi$ is admissible, then 
$$\int_\R\tr(|\Phi(2p)-\Phi(p)|^2)\frac{dp}{|p|}<\infty.$$
\item [(ii)] If $\Phi$ is an even differentiable function satisfying 
$$\int_0^\infty \tr(|\Phi'(p)|^2)pdp<\infty,$$
then $\Phi$ is admissible. 
\end{itemize}
\end{theorem} 

\begin{remark} For an admissible symbol $\Phi$, we call $\hat{\Phi}_0$ in Theorem \ref{Sobolev},(3) the \textit{regular part of $\hat{\Phi}$}. 
It is not clear whether $\hat{\Phi}_0$ gives a distribution on $\R$ in general.  
However, when it is the case, e.g. $\hat{\Phi}_0\in L^1(\R)\otimes M_N(\C)$, then we have 
$\hat{\Phi}=\hat{\Phi}_0+\delta_0\otimes Q$ for some $Q\in M_N(\C)$. 
\end{remark}

\begin{definition}\label{Toeplitz} 
Let $\Phi\in L^\infty(\R)\otimes M_N(\C)$ be an admissible symbol, and let $A=T_\Phi$. 
We denote by $\alpha^\Phi=\{\alpha^\Phi_t\}_{t\geq 0}$ the $E_0$-semigroup acting on the type I factor $\cM_A$ determined by 
$$\alpha^\Phi_t(\pi_A(a(f)))=\pi_A(a(S_tf)),\quad \forall f\in K.$$
We call $\alpha^\Phi$ the \textit{Toeplitz CAR flow} associated with the symbol $\Phi$. 
\end{definition}

For a Toeplitz CAR flow $\alpha^\Phi$, we simply denote $\cE_\Phi:=\cE_{\alpha^\Phi}$ and  
$\cA^\Phi_a(I):=\cA^{\cE_\Phi}_a(I)$.
  
\begin{example} When $\Phi\in M_N(\C)$ is a constant projection, the corresponding Toeplitz CAR flow is nothing but 
the CAR flow of index $N$, 
which gives the unique cocycle conjugacy class of type I$_N$ $E_0$-semigroups. 
\end{example} 

\begin{example}\label{typical examples} 
Powers' first example of a type III $E_0$-semigroup is the Toeplitz CAR flow associated with the symbol 
$$\Phi(p)=\frac{1}{2}\left(
\begin{array}{cc}
1 &e^{i\theta(p)}  \\
e^{-i\theta(p)} &1 
\end{array}
\right),
$$
where $\theta(p)=(1+p^2)^{-1/5}$. 
More generally, if $\theta(p)$ is a real differentiable function satisfying $\theta(-p)=\theta(p)$ for $\forall p\in \R$ and 
$$\int_0^\infty |\theta'(p)|^2pdp<\infty,$$
then Theorem \ref{Sobolev} shows that 
the symbol $\Phi$ as above is admissible. 
In Section \ref{examples}, we will show that for $0<\nu\leq 1/4$, the symbols $\Phi_\nu$, given by $\theta_\nu(p)=(1+p^2)^{-\nu}$ 
in place of $\theta(p)$ above, give rise to mutually non cocycle conjugate type III $E_0$-semigroups. 
\end{example}

We summarize a few facts frequently used in this paper in  the next lemma. 
For a measurable subset $E\subset \R$, we set $K_E=L^2(E,\C^N)$. 
We denote by $P_E$ the projection from $\tK$ onto $K_E$. 
When $I\subset (0,\infty)$, we often regard $P_I$ as an element of $B(K)$. 
For simplicity, we write $K_t=K_{(0,t)}$ and $P_t=P_{(0,t)}$ for $t>0$. 

\begin{lemma}\label{interval} Let $\Phi\in L^\infty(\R)\otimes M_N(\C)$ be an admissible symbol, and let $\hat{\Phi}_0$ be 
the regular part of $\hat{\Phi}$. 
We set $A=T_\Phi$. 
\begin{itemize}
\item [(i)] The relative commutant $\cM_A\cap \alpha^\Phi_t(\cM_A)'$ is $\pi_A^t(\fA(K_t))''$. \medskip
\item [(ii)] Let $I$ and $J$ be mutually disjoint two open sets in $\R$. 
We assume that $I$ and $J$ have only finitely many connected components.  
Then $P_JC_\Phi P_I$ is a Hilbert-Schmidt operator with Hilbert-Schmidt norm  
$$\|P_JC_\Phi P_I\|_\HS^2=\frac{1}{4\pi^2}\int_\R|(J+t)\cap I|\tr(|\hat{\Phi}_0(t)|^2)dt.$$
\item [(iii)] Let $I\subset (0,\infty)$ be an open (finite or infinite) interval. 
Then the restriction of $\pi_{A}$ to $\fA(K_I)$ is of type I, and the commutator $[C_\Phi,P_I]$ is Hilbert-Schmidt. 
\end{itemize}
\end{lemma}

\begin{proof}(i) The statement follows Lemma \ref{commutant},(i). 

(ii) Let $f\in \cD(I,\C^N)$ and $g\in \cD(J,\C^N)$. 
Then 
$$\inpr{C_\Phi f}{g}=\sum_{i,j=1}^N\frac{1}{2\pi}\int_{\R}\Phi(p)_{ij}\hat{f_j}(p)\overline{\hat{g_i}(p)}dp
=\sum_{i,j=1}^N\frac{1}{2\pi}\int_{\R}\Phi(p)_{ij}\widehat{f_j*g_i^{\#}}(p)dp,$$
where $g_i^{\#}(x)=\overline{g_i(-x)}$. 
Since $\widehat{f_j*g_i^{\#}}\in \cD(\R\setminus \{0\})$, we get 
\begin{eqnarray*}
\inpr{C_\Phi f}{g}&=&\sum_{i,j=1}^N\frac{1}{2\pi}\int_{\R}\hat{\Phi}_0(x)_{ij}f_j*g_i^{\#}(x)dx\\
&=&\sum_{i,j=1}^N\frac{1}{2\pi}\int_{\R^2}\hat{\Phi}_0(y-x)_{ij}f_j(y)\overline{g_i(x)}dxdy.
\end{eqnarray*}
Since $\chi_J(x)\chi_I(y)\hat{\Phi}_0(y-x)$ is square integrable (as we will see below), 
the operator $P_JC_\Phi P_I$ is Hilbert-Schmidt, and its Hilbert-Schmidt norm is
\begin{eqnarray*}
\frac{1}{4\pi^2}\int_{\R^2}\chi_J(x)\chi_I(y)\tr(|\hat{\Phi}_0(y-x)|^2)dxdy
&=&\frac{1}{4\pi^2}\int_\R|(J+t)\cap I|\tr(|\hat{\Phi}_0(t)|^2)dt\\
&<&\infty,
\end{eqnarray*}
where we use Theorem \ref{Sobolev},(3). 

(iii) Applying (ii) to $I$ and $J=\R\setminus \overline{I}$, we see that $(1_{\tK}-P_I)C_\Phi P_I$ is Hilbert-Schmidt. 
This and Theorem \ref{qfstate},(ii),(iii) show the first statement. 
Since 
$$[C_\Phi,P_I]=(1_{\tK}-P_I)C_\Phi P_I-P_IC_\Phi (1_{\tK}-P_I),$$
the commutator $[C_\Phi,P_I]$ is Hilbert-Schmidt. 
\end{proof}

\section{A dichotomy theorem}

Based on Powers' argument in \cite{Po1}, Arveson proved the following type III criterion in \cite[Theorem 13.6.1]{Arv}: 

\begin{theorem}[Arveson--Powers]\label{P-A}
Let $\Phi\in L^\infty(\R)\otimes M_N(\C)$ be an admissible symbol having the limit 
$$\Phi(\infty):=\lim_{|p|\to\infty}\Phi(p).$$
If the Toeplitz CAR flow $\alpha^\Phi$ is spatial, then 
$$\int_\R\tr(|\Phi(p)-\Phi(\infty)|^2)dp<\infty.$$
\end{theorem}

The purpose of this section is to generalize Theorem \ref{P-A}, and to show the following dichotomy theorem, 
which can be considered as an analogue of \cite[Theorem 39]{pdct}.  

\begin{theorem}\label{dichotomy}
Let $\Phi\in L^\infty(\R)\otimes M_N(\C)$ be an admissible symbol. 
Then the following conditions are equivalent: 
\begin{itemize}
\item [(i)] The Toeplitz CAR flow $\alpha^\Phi$ is of type I$_N$. 
\item [(ii)] The Toeplitz CAR flow $\alpha^\Phi$ is spatial. 
\item [(iii)] There exists a projection $Q\in M_N(\C)$ satisfying 
$$\int_\R\tr(|\Phi(p)-Q|^2)dp<\infty.$$
\end{itemize}
In particular, every Toeplitz CAR flow is either of type I or type III.
\end{theorem}

The implication from (i) to (ii) is trivial. 
That from (ii) to (iii) is a generalization of Theorem \ref{P-A}. 
Although we follow the same strategy as in the proof of Theorem \ref{P-A}, 
we will make a significant simplification of the argument using Arveson's classification of type I product systems 
(see Lemma \ref{even unit} below), 
which allows us to obtain the statement of this form. 
Since $\alpha^Q$ with a constant projection $Q\in M_N(\C)$ is of type I$_N$, 
the implication from (iii) to (i) follows from a $L^2$-perturbation theorem 
stated below, which can be considered as an analogue of \cite[Theorem 7.4,(1)]{IS}. 

\begin{theorem}\label{L2} Let $\Phi, \Psi\in L^\infty(\R)\otimes M_N(\C)$ be admissible symbols. 
If 
$$\int_\R\tr(|\Phi(p)-\Psi(p)|^2)dp<\infty,$$
then $\alpha^\Phi$ and $\alpha^\Psi$ are cocycle conjugate.  
\end{theorem}

We first give a representation theoretical consequence of the above square integrability condition. 

\begin{lemma}\label{local equivalence} Let $\Phi, \Psi\in L^\infty(\R)\otimes M_N(\C)$ be admissible symbols. 
We set $A=T_\Phi$ and $B=T_\Psi$. 
Then 
$$\int_\R\tr(|\Phi(p)-\Psi(p)|^2)dp<\infty,$$
if and only if for any (some) non-degenerate finite interval $I\subset (0,\infty)$, 
the two quasi-free states $\omega_{P_IAP_I}$ and $\omega_{P_IBP_I}$ of $\fA(K_I)$ are quasi-equivalent.
\end{lemma}

\begin{proof} Thanks to Lemma \ref{typeI}, the two states $\omega_{P_IAP_I}$ and $\omega_{P_IBP_I}$ 
are quasi-equivalent if and only if the following quantity is finite: 
\begin{eqnarray*}
\lefteqn{\tr\big(P_IC_\Phi P_I C_{1-\Psi}P_I C_\Phi P_I+P_IC_{1-\Phi}P_IC_{\Psi}P_IC_{1-\Phi}P_I\big)
} \\
 &=&\tr\big(C_{1-\Psi}(P_IC_\Phi P_I)^2  C_{1-\Psi}+C_{\Psi}(P_IC_{1-\Phi}P_I)^2C_{\Psi}\big).
\end{eqnarray*}
Since $P_IAP_I-(P_IAP_I)^2$ and $P_IBP_I-(P_IBP_I)^2$ are trace class operators 
(see Theorem \ref{qfstate},(iii) and Lemma \ref{interval},(iii)), 
we can replace $(P_IC_\Phi P_I)^2$ with $P_IC_\Phi P_I$ and 
$(P_IC_{1-\Phi}P_I)^2$ with $P_IC_{1-\Phi}P_I$ in the above formula, and we get  
$$\tr\big(C_{1-\Psi}P_IC_\Phi P_IC_{1-\Psi}+C_{\Psi}P_IC_{1-\Phi}P_IC_{\Psi}\big)
=\|C_\Phi P_IC_{1-\Psi}\|_{\HS}^2+\|C_{1-\Phi}P_IC_{\Psi}\|_{\HS}^2.$$
Since the commutators $[C_{1-\Psi},P_I]$ and $[C_\Phi,P_I]$ are Hilbert-Schmidt 
(see Lemma \ref{interval},(iii)), the right-hand side is finite if and only if 
$$\|C_{\Phi(1-\Psi)}P_I\|_{\HS}^2+\|C_{(1-\Phi)\Psi}P_I\|_{\HS}^2$$
is finite. 
\cite[Proposition 13.4.1]{Arv} shows that this is equal to 
$$\frac{|I|}{2\pi}\int_\R\tr(|\Phi(p)-\Psi(p)|^2)dp,$$
and we get the statement. 
\end{proof}

\begin{proof}[Proof of Theorem \ref{L2}]
Assume that $\Phi-\Psi$ is square integrable. 
Let $A=T_\Phi$ and $B=T_\Psi$. 
We will apply Theorem \ref{isomorphic} to $E=\cE_\Phi$ and $F=\cE_\Psi$, and show that $\alpha^\Phi$ and $\alpha^\Psi$ are cocycle conjugate. 

Thanks to Lemma \ref{local equivalence}, the two representations $\pi_A$ and $\pi_B$ are quasi-equivalent 
when they are restricted to $\fA(K_1)$. 
This implies that there exists an isomorphism $\rho_0$ from $\pi_A(\fA(K_1))''$ onto 
$\pi_B(\fA(K_1))''$ satisfying $\rho_0(\pi_A(a(f)))=\pi_B(a(f))$ for $\forall f\in K_1$. 
Since $\rho_0$ preserves the grading, we may assume $\rho_0(R^A_{K_1})=R^B_{K_1}$ by replacing $R^B_{K_1}$ with 
$-R^B_{K_1}$ if necessary. 
We may also assume $R^A=R^A_{K_1}R^A_{K_{(1,\infty)}}$ and $R^B=R^B_{K_1}R^B_{K_{(1,\infty)}}$.   

We claim that $\rho_0$ extends to an isomorphism $\rho_1$ from $(\pi_A(\fA(K_1))\cup\{R^A\})''$ onto $(\pi_B(\fA(K_1))\cup\{R^B\})''$ 
satisfying $\rho_1(R^A)=R^B$. 
Indeed, since $R^A_{K_{(1,\infty)}}$ commutes with $\pi_A(\fA(K_1))$, we have 
$$(\pi_A(\fA(K_1))\cup\{R^A\})''=\frac{1+R^A_{K_{(1,\infty)}}}{2}\pi_A(\fA(K_1))''\oplus \frac{1-R^A_{K_{(1,\infty)}}}{2} 
\pi_A(\fA(K_1))''.$$
For the same reason, 
$$(\pi_B(\fA(K_1))\cup\{R^B\})''=\frac{1+R^B_{K_{(1,\infty)}}}{2}\pi_B(\fA(K_1))''\oplus \frac{1-R^B_{K_{(1,\infty)}}}{2} 
\pi_B(\fA(K_1))'',$$
and so $\rho_0$ extends to $\rho_1$ satisfying $\rho_1(R^A_{K_{(1,\infty)}})=\rho_1(R^B_{K_{(1,\infty)}})$. 
In consequence, we have $\rho_1(R^A)=R^B$. 

Let $\rho$ be the restriction of $\rho_1$ to $\cM_A\cap \alpha^\Phi_1(\cM_A)'$, which is identified with $B(\cE_\Phi(1))$. 
Thanks to Lemma \ref{commutant},(i), it is generated by $\{\pi_A(a(f))R^A;\; f\in K_1\}.$   
Then the image of $\rho$ is generated by $\{\pi_B(a(f))R^B;\; f\in K_1\},$
and so it is $\cM_B\cap \alpha^{\Psi}_1(\cM_B)'$, which is identified with $B(\cE_\Psi(1))$. 
In the same way, we can see that $\rho$ satisfies $\rho(\cA^\Phi(0,s))=\cA^\Psi(0,s)$ for any $0\leq s\leq 1$. 
Thus we get the statement from Theorem \ref{isomorphic}. 
\end{proof}

Now we start the proof of the implication (ii) $\Rightarrow$ (iii) in Theorem \ref{dichotomy}. 
Recall that $\gamma$ is the grading automorphism $\gamma(\pi_A(a(f)))=-\pi_A(a(f))$. 

\begin{lemma}\label{even unit} Let $\Phi\in L^\infty(\R)\otimes M_N(\C)$ be an admissible symbol. 
If $\alpha^\Phi$ is spatial, then there exists a unit $V=\{V_t\}_{t>0}$ for $\alpha^{\Phi}$ satisfying 
$\gamma(V_t)=V_t$ for $\forall t>0$. 
\end{lemma}

\begin{proof} Since $\gamma$ commutes with $\alpha^\Phi_t$ for $\forall t>0$, 
it induces an automorphism of the corresponding product system $\cE_\Phi$. 
When $\cE_\Phi$ is of type II$_0$, it is easy to show the statement, and so we assume the index of $\cE_\Phi$ is not $0$. 
Let $E$ be the subproduct system of $\cE_{\alpha^{\Phi}}$ generated by the units, and let $\beta$ be the automorphism of 
$E$ induced by $\gamma$. 
Then the statement follows from the following claim: 
for any period two automorphism $\beta$ of any type I product system $E$, there exists a unit of $E$ fixed by $\beta$. 
Note that the type I product systems are completely classified, and the action of $\Aut(E)$ on the set of units $\cU_E$ 
is well-known (see \cite[Section 3.8]{Ar}). 

Let $L$ be a Hilbert space whose dimension is the same as the index of $E$, and let $\cU(L)$ be the unitary group of $L$.  
Then $\Aut(E)$ is identified with $G_L=\R\times L \times \cU(L)$ having the group operation  
$$(\lambda,\xi,U)(\mu,\eta,V)=(\lambda+\mu+\IM\inpr{\xi}{U\eta},\xi+U\eta,UV).$$
The set $\cU_E$ together with the $\Aut(E)$-action on it is identified with $\C\times L$ with the $G_L$-action 
$$(\lambda,\xi,U)\cdot(a,\eta)=(a+i\lambda-\frac{\|\xi\|^2}{2}-\inpr{U\eta}{\xi},\xi+U\eta).$$
Any element $g\in G_L$ of order two is of the form $g=(0,\xi,U)$ with $U^2=1$ and $U\xi=-\xi$. 
Now we can see that $(0,\frac{1}{2}\xi)\in \C\times L$ is fixed by $g$. 
\end{proof}

The following lemma is a slight generalization of \cite[Lemma 4.5]{Po1} and \cite[Lemma 13.6.5]{Arv}. 
For later use, we will show a little stronger statement than we need in this section. 

\begin{lemma}\label{making a state} Let $\Phi\in L^\infty(\R)\otimes M_N(\C)$ be an admissible symbol and $A=T_\Phi$. 
If $V\in \cE_\Phi(t)$ is a normalized vector satisfying $\gamma(V)=\pm V$, 
then there exists a pure $\gamma$-invariant state $\varphi$ of $\fA(K_t)$ such that 
$V^*\pi_A(X)V=\varphi(X)1$ for any $X\in \fA(K_t)$. 
\end{lemma}

\begin{proof} Throughout the proof, the symbol $a^\dagger(f)$ means either $a(f)$ or $a(f)^*$. 
Let $f_1,f_2, \cdots, f_n\in K_t$, and $X=a^\dagger(f_1)a^\dagger(f_2)\cdots a^\dagger(f_n)$. 
Then for any $g\in K$,  we have 
\begin{eqnarray*}
V^*\pi_A(X)V\pi_A(a^\dagger(g))&=&V^*\pi_A(Xa^\dagger(S_tg))V=(-1)^nV^*\pi_A(a^\dagger(S_tg)X)V\\
&=&(-1)^n\pi_A(a^\dagger(g))V^*\pi_A(X)V.
\end{eqnarray*}
If $n$ is even, this shows that $V^*\pi_A(X)V$ is in the center $Z(\cM_A)$ of $\cM_A$, and so it is a scalar. 
If $n$ if odd, the operator $R^AV^*\pi_A(X)V$ is a scalar for the same reason, and on the other hand, it is an odd 
operator with respect to $\gamma$. 
Thus $V^*\pi_A(X)V=0$, which shows that there exists a $\gamma$-invariant state $\varphi$ such that $V^*\pi_A(X)V=\varphi(X)1$ 
for all $X\in \fA(K_t)$. 

It only remains to show that $\varphi$ is pure. 
Recall that the twisted representation $\pi_A^t$ is defined by $\pi_A^t(a(f))=i\pi_A(a(f))R^A$, and 
$\cM_A\cap \alpha^\Phi_t(\cM_A)'=\pi_A^t(\fA(K_t))''$. 
We denote by $\pi$ the irreducible representation of $\fA(K_t)$ on $\cE_\Phi(t)$ given by 
$\pi(X)=\sigma(\pi_A^t(X))$ on $\cE_\Phi(t)$, where $\sigma(Y)$ denotes the left multiplication of $Y$. 
Then the pure state of $\fA(K_t)$ given by $X\mapsto \inpr{\pi(X)V}{V}=V^*\pi_A^t(X)V$ 
coincides with $\varphi$ because both $\varphi$ and this state are $\gamma$-invariant, 
and $\pi_A$ and $\pi_A^t$ coincide on the even part of $\fA(K_t)$. 
\end{proof}

\begin{proof}[Proof of (ii) $\Rightarrow (iii)$ in Theorem \ref{dichotomy}] 
Let $\Phi\in L^\infty(\R)\otimes M_N(\C)$ be an admissible symbol, and let $A=T_\Phi$. 
Assume that $\alpha^\Phi$ is spatial. 
Then Lemma \ref{even unit} shows that there exists a normalized unit $V=\{V_t\}_{t\geq 0}$ satisfying $\gamma(V_t)=V_t$ for all $t$. 
Let $\varphi$ be the state of $\fA(K_1)$ defined by 
$\varphi(X)=\inpr{\pi_A(X)V_1\Omega_A}{V_1\Omega_A}$ for $X\in \fA(K_1)$, and let $B\in B(K_1)$ the covariance operator for $\varphi$. 
Then Lemma \ref{making a state} shows that $V_t^*\pi_A(a(f))V_t=0$ for any $f\in K_t$.  
We claim that there exists a positive contraction $Q\in L^\infty((0,1))\otimes M_N(\C)$ such that $B$ is the multiplication operator of $Q$. 
To prove the claim, it suffices to show that $B$ commutes with $P_t$ for all $0< t<1$. 
Indeed, if $f\in K_t$ and $g\in K_{(t,1)}$, then 
$$V_1^*\pi_A(a(f)a(g)^*)V_1=V_{1-t}^*V_t^*\pi_A(a(f))V_t\pi_A(a(S^*_tg)^*)V_{1-t}=0.$$
Thus we get $P_{(t,1)}BP_t=0$, and the claim is shown. 

Note that $\varphi$ is quasi-equivalent to $\omega_{P_1AP_1}$. 
We claim that $B$ is a projection. 
Let $\K(K_1)$ be the set of compact operators of $K_1$, and let $q:B(K_1)\to B(K_1)/\K(K_1)$ be the quotient map. 
Then thanks to Lemma \ref{covariance operator},(ii), we have $q(P_1AP_1)=q(B)$. 
Since $\omega_{P_1AP_1}$ is a type I state, we have $q(P_1AP_1)^2=q(P_1AP_1)$, and so $B-B^2$ is a compact operator. 
This is possible only if $Q(x)$ is a projection for almost every $x\in (0,1)$, and so $B$ is a projection. 

Since $B$ is a projection, Lemma \ref{covariance operator},(i) implies $\varphi=\omega_B$. 
Since $\omega_{P_1AP_1}$ and $\omega_B$ are quasi-equivalent, Theorem \ref{qfstate},(v) implies 
$$\|C_\Phi(P_1-B)\|_\HS^2+\|C_{1-\Phi}B\|_\HS^2=\tr\big((P_1-B)C_\Phi(P_1-B)+BC_{1-\Phi}B\big)<\infty.$$
A similar computation as in \cite[Proposition 13.4.1]{Arv} shows that the left-hand side is  
$$\frac{1}{2\pi}\int_0^1\int_\R\tr(|\Phi(p)-Q(x)|^2)dpdx.$$ 
Thus the integral 
$$\int_\R\tr(|\Phi(p)-Q(x)|^2)dp$$ 
is finite for almost every $x\in (0,1)$, and the proof is finished. 
\end{proof}

\begin{example}\label{nonconverging}
Let $\theta(p)$ be a real smooth function satisfying $\theta(-p)=\theta(p)$ for all $ p\in \R$ and 
$\theta(p)=\log(\log |p|)$ (or $\theta(p)=\log^\alpha|p|$ with $0<\alpha<1/2$) for large $|p|$. 
Then $\Phi$ associated with $\theta$ in Example \ref{typical examples} is an admissible symbol without having 
limit at infinity. 
While Theorem \ref{P-A} does not apply to such $\Phi$, now we know from Theorem \ref{dichotomy} that the Toeplitz CAR flow $\alpha^{\Phi}$ is 
of type III. 
\end{example}

\section{Type I factorizations associated with Toeplitz CAR flows}\label{factorization}
Thanks to Theorem \ref{dichotomy}, we have a complete understanding of spatial Toeplitz CAR flows now. 
The purpose of this section is to calculate the invariant we introduced in Subsection \ref{invariant for type III} 
in the case of type III Toeplitz CAR flows. 

\begin{theorem} \label{CABATIF} Let $\Phi\in L^\infty(\R)\otimes M_N(\C)$ be an admissible symbol, and let $\{a_n\}_{n=0}^\infty$ be a 
strictly increasing sequence of non-negative numbers such that  $a_0=0$ and it converges to a finite number $a$. 
Let $I_n=(a_n,a_{n+1})$ and $O=\bigcup_{n=0}^\infty I_{2n}$. 
\begin{itemize}
\item [(i)] If 
$$\sum_{n=0}^\infty\|(1_{\tK}-P_{I_n})C_\Phi P_{I_n}\|_\HS^2<\infty,$$
then $\{\cA^\Phi_a(I_n)\}_{n=1}^\infty$ is a CABATIF. 
\item [(ii)]  If $\{\cA^\Phi_a(I_n)\}_{n=0}^\infty$ is a CABATIF, then 
$\|(1_{\tK}-P_O)C_\Phi P_O\|_\HS^2<\infty$.  
\end{itemize}
\end{theorem}

We prepare a few facts used in the proof of (i) first. 

\begin{lemma}\label{PQ} Let $H$ be a Hilbert space, and let $P,Q\in B(H)$ be projections. 
Then  
$$\|(1-P)QP\|_\HS=\|(1-Q)PQ\|_\HS.$$
\end{lemma}

\begin{proof} There is a decomposition of $H$ into closed subspaces (each subspace could possibly be $\{0\}$) 
$$H=H_1\oplus H_2\oplus H_3\oplus H_4\otimes \C^2\oplus H_5$$
such that the two projections are expressed as 
$$P=1_{H_1}\oplus 1_{H_2}\oplus 0\oplus 
\left(\begin{array}{cc}
1_{H_4} &0  \\
0 &0 
\end{array}\right)
\oplus 0,
$$
$$Q=1_{H_1}\oplus 0\oplus 1_{H_3}\oplus 
\left(\begin{array}{cc}
c^2 &cs  \\
cs &s^2 
\end{array}\right)
\oplus 0,
$$
where $c$ and $s$ are non-singular positive contractions satisfying $c^2+s^2=1_{H_4}$ (see \cite[p.308]{Ta}). 
Then we have 
$$\|(1-P)QP\|_\HS^2
=\|\left(
\begin{array}{cc}
0 &0  \\
cs &0 
\end{array}
\right)\|_\HS^2
=\tr(c^2s^2),$$
\begin{eqnarray*}
\|(1-Q)PQ\|_\HS^2&=&
\|\left(
\begin{array}{cc}
s^2 &-cs  \\
-cs &c^2 
\end{array}
\right)
\left(
\begin{array}{cc}
1 &0  \\
0 &0 
\end{array}
\right)
\left(
\begin{array}{cc}
c^2 &cs  \\
cs &s^2 
\end{array}
\right)
\|_\HS^2 \\
 &=&\|\left(
\begin{array}{cc}
c^2s^2 &cs^3  \\
-c^3s &-c^2s^2 
\end{array}
\right)\|_\HS^2
=\tr\big(2c^4s^4+c^2s^6+c^6s^2\big)\\
 &=&\tr\big(c^2s^2(c^2+s^2)^2\big)
 =\tr(c^2s^2).  
\end{eqnarray*}
\end{proof}

\begin{lemma}\label{qe and type I} Let the notation be as in Theorem \ref{CABATIF}, and let $A=T_\Phi$. 
We set 
$$B=\sum_{n=0}^\infty P_{I_n}AP_{I_n}+P_{(a,\infty)}AP_{(a,\infty)}.$$
Then the following conditions are equivalent: 
\begin{itemize}
\item [(i)] The assumption of Theorem \ref{CABATIF},(i) holds. 
\item [(ii)] The quasi-free state $\omega_B$ is of type I. 
\item [(iii)] The two quasi-free states $\omega_A$ and $\omega_B$ are quasi-equivalent. 
\end{itemize}
\end{lemma}

\begin{proof} Theorem \ref{qfstate},(iii) and Lemma \ref{interval},(iii) imply that (i) and (ii) are equivalent. 
We show the equivalence of (i) and (iii). 
We set $I_{-1}=(a,\infty)$. 
Since $\omega_A$ is of type I, Lemma \ref{typeI} shows that $\omega_A$ and $\omega_B$ are quasi-equivalent if and only if 
the following quantity is finite: 
\begin{eqnarray*}
\lefteqn{\tr\big(A(1_K-B)A+(1_K-A)B(1_K-A)\big)}\\
&=&\sum_{n=-1}^\infty \tr\big(P_+C_\Phi P_{I_n}C_{1-\Phi}P_{I_n}C_\Phi P_+ + P_+C_{1-\Phi}P_{I_n}C_\Phi P_{I_n}C_{1-\Phi}P_+\big) \\
&=&\sum_{n=-1}^\infty (\| C_{1-\Phi}P_{I_n}C_\Phi P_+\|_\HS^2+\|C_\Phi P_{I_n}C_{1-\Phi}P_+\|_\HS^2).
\end{eqnarray*}
Note that since 
\begin{eqnarray*}\lefteqn{
\sum_{n=-1}^\infty (\| C_{1-\Phi}P_{I_n}C_\Phi (1_{\tK}-P_+)\|_\HS^2+\|C_\Phi P_{I_n}C_{1-\Phi}(1_{\tK}-P_+)\|_\HS^2)}\\
&\leq&\sum_{n=-1}^\infty\tr\big((1_{\tK}-P_+)C_\Phi P_{I_n}C_\Phi (1_{\tK}-P_+)
+(1_{\tK}-P_+)C_{1-\Phi}P_{I_n}C_{1-\Phi}(1_{\tK}-P_+)\big)\\
&=& 
\|P_+ C_\Phi (1_{\tK}-P_+)\|_\HS^2+\|P_+ C_{1-\Phi}(1_{\tK}-P_+)\|_\HS^2\\
&=&2\|P_+ C_\Phi (1_{\tK}-P_+)\|_\HS^2<\infty,
\end{eqnarray*}
the above quantity is finite if and only if 
$$\sum_{n=-1}^\infty \| C_{1-\Phi}P_{I_n}C_\Phi \|_\HS^2<\infty.$$
Thanks to Lemma \ref{PQ}, this is equivalent to 
$$\sum_{n=-1}^\infty \| (1_{\tK}-P_{I_n})C_\Phi P_{I_n}\|_\HS^2<\infty.$$
Since $\| (1_{\tK}-P_{I_{-1}})C_\Phi P_{I_{-1}}\|_\HS^2<\infty$, we conclude that (i) is equivalent to (iii). 
\end{proof}

\begin{proof}[Proof of Theorem \ref{CABATIF},(i)] Assume that the assumption of Theorem \ref{CABATIF},(i) holds. 
It suffices to show that for any strictly increasing sequence of non-negative integers 
$\{n_{m}\}_{m=0}^\infty$, the von Neumann algebra $\cA^\Phi_a(E):=\bigvee_{m=0}^\infty \cA^\Phi_a(I_{n_m})$ is a type I factor, 
where $E=\bigcup_{m=0}^\infty I_{n_m}$. 
Note that $\cA^\Phi_a(E)$ is always a factor (see \cite[Remark 8.2]{IS}).
We may assume $n_0=0$ without loss of generality. 
Identifying $B(\cE^\Phi(a))$ with $\cM_A\cap \alpha^\Phi_a(\cM_A)'$, we see that it suffices to show 
the factor  
$$\bigvee_{m=0}^\infty \alpha^\Phi_{a_{n_m}}(\cM_A\cap \alpha^\Phi_{a_{n_m+1}-a_{n_m}}(\cM_A)')$$
is of type I. 
Recall that we have $\cM_A\cap \alpha^\Phi_{t}(\cM_A)'=\pi_A^t(\fA(K_t))''$, where $\pi_A^t(a(f))=i\pi_A(a(f))R^A$. 
Since 
$$\alpha^\Phi_t(R^A)=\pm R^A_{K_{(t,\infty)}}=\pm \epsilon_{K_t,K_{(t,\infty)}}R^AR^A_{K_t},$$ 
we get 
$$\alpha^\Phi_{a_{n_m}}(\cM_A\cap \alpha^\Phi_{a_{n_m+1}-a_{n_m}}(\cM_A)')
=\{\pi_A^t(a(f))R^A_{K_{a_{n_m}}},R^A_{K_{a_{n_m}}}\pi_A^t(a(f))^*;\; f\in K_{I_{n_m}}\}''.$$
Thanks to Lemma \ref{commutant},(ii), it suffices to show that the factor 
$$\bigvee_{m=0}^\infty\{\pi_A(a(f))R^A_{K_{a_{n_m}}},R^A_{K_{a_{n_m}}}\pi_A(a(f))^*;\; f\in K_{I_{n_m}}\}''$$
is of type I. 

Let 
$$\cN_m:=\bigvee_{k=0}^m\{\pi_A(a(f))R^A_{K_{a_{n_k}}},R^A_{K_{a_{n_k}}}\pi_A(a(f))^*;\; f\in K_{I_{n_k}}\}'' ,$$ 
and let $J_m=\bigcup_{k=1}^m(a_{n_{k-1}+1},a_{n_k})$. 
Since 
$$R^A_{K_{a_{n_m}}}=\pm R^A_{K_{J_m}}\prod_{k=0}^{m-1} R^A_{K_{I_{n_k}}},$$
and 
$$R^A_{I_{n_k}}\in \{\pi_A(a(f))R^A_{K_{J_k}},\; R^A_{K_{J_k}}\pi_A(a(f))^*;\; f\in K_{I_{n_k}}\}'',$$
we can show 
$$\cN_m=\bigvee_{k=0}^m\{\pi_A(a(f))R^A_{K_{J_k}},R^A_{K_{J_k}}\pi_A(a(f))^*;\; f\in K_{I_{n_k}}\}''$$
by induction, where we use the convention $R^A_{K_{J_0}}=1$. 
Thus to prove the statement, it suffices to show that the factor 
$$\bigvee_{m=0}^\infty\{\pi_A(a(f))R^A_{K_{J_m}},R^A_{K_{J_m}}\pi_A(a(f))^*;\; f\in K_{I_{n_m}}\}''$$
is of type I. 

Let $B$ be as in Lemma \ref{qe and type I}. 
Since $\pi_A$ and $\pi_B$ are quasi-equivalent, there exists an isomorphism $\theta$ from 
$\cM_A$ onto $\cM_B$ satisfying $\theta(\pi_A(f))=\theta(\pi_B(a(f)))$ for any $f\in K$. 
Since $\theta$ preserves the grading, we may assume $\theta(R^A_{K_I})=R^B_{K_I}$ for any interval $I\subset (0,\infty)$. 
Thus to prove the statement, it suffices to show that the factor 
$$\cN:=\bigvee_{m=0}^\infty\{\pi_B(a(f))R^B_{K_{J_m}},R^B_{K_{J_m}}\pi_B(a(f))^*;\; f\in K_{I_{n_m}}\}''$$
is of type I. 

Since $J_m$ is disjoint from $E$, the self-adjoint unitary $R^B_{K_{J_m}}$ commutes with any $\pi_B(a(f))$ with $f\in K_E$. 
Thus $\cN$ is generated by the factor representation $\pi$ of $\fA(K_E)$ determined by 
$\pi(a(f))=\pi_B(a(f))R^B_{J_m}$ for $f\in K_{I_{n_m}}$. 
Let $\omega$ be the state of $\fA(K_E)$ defined by $\omega(X):=\inpr{\pi(X)\Omega_B}{\Omega_B}$ for $X\in \fA(K_E)$. 
Since $\pi$ is a factor representation, the GNS representation of $\omega$ is quasi-equivalent to $\pi$. 

We claim that $\omega$ coincides with $\omega_{P_EBP_E}$.  
Let $X_i\in \fA(K_{I_{n_i}})$, $i=0,1,\cdots, m$ be of the form  
$$X_i=a^\dagger(f^i_1)a^\dagger(f^i_2)\cdots a^\dagger(f^i_{l_i}),\quad$$
with $f^{i}_j\in K_{I_{n_i}}$, where $a^\dagger(f)$ means either $a(f)$ or $a(f)^*$. 
Then we have $\pi(X_i)=\pi_B(X_i){R^B_{K_{J_i}}}^{l_i},$
and 
$$\omega(X_1X_2\cdots X_m)=\inpr{\pi_B(X_1X_2\cdots X_m)Y\Omega_B}{\Omega_B}$$
where $Y$ is an element in the even part of $\pi_B(\fA(K_E^\perp))''$. 
Since $B$ commutes with $P_{I_n}$ for any $n$, if 
one of $l_1,l_2,\cdots, l_m$ is odd, then approximating $Y$ by polynomials of $\pi_B(a^\dagger(f))$ with 
$f\in K_E^\perp$, we see that the right-hand side is 0  (consider the contributing 2-point functions). 
When $l_1,l_2,\cdots,l_m$ are all even, we have 
$$\omega(X_1X_2\cdots X_m)=\inpr{\pi_B(X_1X_2\cdots X_m)\Omega_B}{\Omega_B}=\omega_B(X_1X_2\cdots X_m),$$
which shows $\omega=\omega_{P_EBP_E}$. 
Thus to prove the statement, it suffices to show that $\omega_{P_EBP_E}$ is of type I. 

Since $P_E$ commutes with $B$, we get 
$$\tr\big(P_EBP_E-(P_EBP_E)^2\big)=\tr\big(P_E(B-B^2)\big)\leq \tr(B-B^2).$$
Now the statement follows from Theorem \ref{qfstate},(iii) and Lemma \ref{qe and type I}. 
\end{proof}

We proceed to the proof of Theorem \ref{CABATIF},(ii). 

\begin{lemma}\label{product state} Let $L_n$, $n=0,1,\cdots,$ be Hilbert spaces, and let $L=\bigoplus_{n=0}^\infty L_n$. 
Assume that $\varphi$ is a $\gamma$-invariant state of $\fA(L)$ satisfying the following two conditions:
\begin{itemize}
\item [(i)] For any natural number $n$ and $X_i\in \fA(L_i)$, $i=0,1,\cdots,n$, 
$$\varphi(X_1X_2\cdots X_n)=\varphi(X_1)\varphi(X_2)\cdots \varphi(X_n).$$
\item [(ii)] The restriction $\varphi_n$ of $\varphi$ to $\fA(L_n)$ is a pure state for any $n$. 
\end{itemize}
Then $\varphi$ is a pure state.  
\end{lemma}

\begin{proof} 
Let $(H_n,\pi_n,\Omega_n)$ be the GNS triple of $\varphi_n$, and let 
$H=\bigotimes_{n=0}^\infty{}^{(\otimes \Omega_n)}H_n$ be the ITPS of the Hilbert spaces $\{H_n\}_{n=0}^\infty$ with respect to 
the reference vectors $\{\Omega_n\}_{n=0}^\infty$. 
We set $\Omega=\bigotimes_{n=0}^\infty \Omega_n$. 
Since $\varphi_n$ is a $\gamma$-invariant state of $\fA(L_n)$, there exists a self-adjoint unitary $R_n\in B(H_n)$ satisfying 
$R_n\pi_n(X)\Omega_n=\pi_n(\gamma(X))$ for all $X\in \fA(L_n)$. 
We introduce a representation $\pi$ of $\fA(L)$ on $H$ by setting $\pi(a(f))$ for $f\in L_n$ as 
$$\pi(a(f))=\left\{
\begin{array}{ll}
\pi_0(a(f))\otimes 1_{\bigotimes_{k=1}^\infty H_k} , &\quad n=0 \\
R_0\otimes R_1\otimes \cdots \otimes R_{n-1}\otimes  \pi_n(a(f))\otimes 1_{\bigotimes_{k=n+1}^\infty H_k}, 
&\quad n>0
\end{array}
\right..$$
Then $\pi$ is irreducible, and the pure state $\psi$ of $\fA(L)$ defined by $\psi(X)=\inpr{\pi(X)\Omega}{\Omega}$ 
satisfies the two conditions (i) and (ii). 
Moreover, the restriction of $\psi$ to $\fA(L_n)$ coincides with $\varphi_n$. 
Since $\{\varphi_n\}_{n=0}^\infty$ and the condition (i) uniquely determine $\varphi$, we conclude that 
$\varphi=\psi$ and it is a pure state.  
\end{proof}

\begin{lemma}\label{decomposable vector} If the assumption of Theorem \ref{CABATIF},(ii) holds, then 
there exist normalized vectors $V\in \cE_\Phi(a)$, $V_n\in \cE_\Phi(a_{n+1}-a_n)$, and $W_n\in \cE_\Phi(a-a_{n+1})$, $n=0,1,2,\cdots,$ 
such that $V$ is factorized as $V=V_0V_1V_2\cdots V_nW_n$, and $\gamma(V_n)=\pm V_n$, $\gamma(W_n)=\pm W_{n}$ 
for any non-negative integer $n$. 
\end{lemma}

\begin{proof}
Assume that $\{\cA^\Phi_a(I_n)\}_{n=0}^\infty$ is a CABATIF. 
Then thanks to Theorem \ref{ArakiWoods}, there exists a sequence of Hilbert spaces with normalized vectors 
$\{(H_n,\xi_n)\}_{n=0}^\infty$ and unitary $U$ from ITPS $H:=\bigotimes_{n=0}^\infty {}^{\otimes \xi_n}H_n$ onto 
$\cE_\Phi(a)$ such that $U\cM_nU^*=\cA^\Phi_a(I_n)$, where 
$$\cM_n=B(H_n)\otimes \C1_{\bigotimes_{m\neq n}H_m}.$$

For $X\in \cM_A\cap \alpha^\Phi_a(\cM_A)'$, we denote by $\sigma(X)\in B(\cE_\Phi(a))$ the corresponding left multiplication operator. 
We claim that for any $0<t<a$, there exists $\epsilon_t\in \{1,-1\}$ such that 
$\gamma(X)=\epsilon_t\sigma(R^A_{K_t})X$ for any $X\in \cE^\Phi(t)$. 
Indeed, let $\epsilon_t$ be the constant determined by $\alpha^\Phi_t(R^A)=\epsilon_t R^A_{K_t}R^A$. 
Then 
$$\gamma(X)=R^AX{R^A}^*=R^A\alpha^\Phi_t(R^A)^*X=\epsilon_t R^A_{K_t}X,$$
which shows the claim. 

The claim (or Lemma \ref{commutant}) implies that for any $X\in \cM_A\cap \alpha^\Phi_a(\cM_A)'$ we have 
$R^A_{K_a}X{R^A_{K_a}}^*=\gamma(X)$. 
Thus $\sigma(R^A_{K_a})$ is a self-adjoint unitary satisfying 
$$\sigma(R^A_{K_a})\cA^\Phi_a(I_n)\sigma(R^A_{K_a})^*=\cA^\Phi_a(I_n).$$
For the same reason, the operator $\sigma(R^A_{K_{I_n}})$ is a self-adjoint unitary in $\cA^\Phi_a(I_n)$ satisfying 
\begin{equation}\label{eq in proof of (ii)}
\sigma(R^A_{K_a})X\sigma(R^A_{K_a})^*=\sigma(R^A_{K_{I_n}})X\sigma(R^A_{K_{I_n}})^*,\quad \forall X\in \cA^E_a(I_n).
\end{equation}

Applying Lemma \ref{fixed product vector} to the self-adjoint unitary $R=U^*\sigma(R^A_{K_a})U\in B(H)$, 
we get a product vector $\eta=\bigotimes_{n=1}^\infty\eta_n\in H$ and self-adjoint unitaries $R_n\in B(H_n)$ 
satisfying the three conditions in the conclusion of Lemma \ref{fixed product vector}. 
We may assume $\|\eta_n\|=1$ by normalizing each $\eta_n$. 
We set $V:=U\eta$. 
Then we have $\gamma(V)=\epsilon_a \sigma(R^A_{K_a})V=\pm V$. 

Let $e_n\in \cM_n$ be the minimal projection satisfying $e_n\eta=\eta$ for all $n$, and set $f_n=Ue_nU^*$, which 
is a minimal projection of $\cA^\Phi_a(I_n)$. 
Then we have $f_nV=V$ for all $n$. 
For each $n$, we can choose a normalized vector $V_n\in \cE_\Phi(a_{n+1}-a_n)$ so that for any $X\in  \cE_\Phi(a_n)$ and 
$Y\in  \cE_\Phi(a-a_{n+1})$ we have $f_n(XV_nY)=XV_nY$. 
Since the self-adjoint unitary $U\rho_n(R_n)U^*\in \cA^\Phi_a(I_n)$ satisfies the same equation as 
(\ref{eq in proof of (ii)}) in place of $\sigma(R^A_{K_{I_n}})$, 
we have either $U\rho_n(R_n)U^*=\sigma(R^A_{K_{I_n}})$ or $UR_nU^*=-\sigma(R^A_{K_{I_n}})$. 
Thus $\rho_n(R_n)e_n\rho_n(R_n)^*=e_n$ implies $\sigma(R^A_{K_{I_n}})f_n\sigma(R^A_{K_{I_n}})^*=f_n$. 
Since $f_n$ is a minimal projection of $\cA^\Phi_a(I_n)$ and $\sigma(R^A_{K_{I_n}})\in \cA^\Phi_a(I_n)$ 
is a self-adjoint unitary, this shows that $XV_nY$ is an eigenvector of 
$\sigma(R^A_{K_{I_n}})$ and $R^A_{K_{I_n}}XV_nY=\pm XV_nY$. 
On the other hand, since $\alpha^\Phi_{a_n}(R^A_{a_{n+1}-a_n})=\pm R^A_{I_n}$ and 
$$R^A_{K_{I_n}}XV_nY=X({\alpha^\Phi_{a_n}}^{-1}(R^A_{I_n})V_n)Y,$$
we see that $V_n$ is an eigenvector of $\sigma(R^A_{a_{n+1}-a_n})$. 
Thus we get $\gamma(V_n)=\pm V_n$. 
Letting  $W_n=(V_1V_2\cdots V_n)^*V$, we finish the proof.
\end{proof}

\begin{proof}[Proof of Theorem \ref{CABATIF},(ii)] 
Since $\|(1_{\tK}-P_O)C_\Phi P_O\|_\HS^2=\tr\big(P_OAP_O-(P_OAP_O)^2\big)$, 
it suffices to show that the restriction of $\pi_A$ to $\fA(K_O)$ is a type I representation 
thanks to Theorem \ref{qfstate}(ii),(iii). 

Let $L_n=K_{a_{2n+1}-a_{2n}}$, and let $L=\bigoplus_{n=0}^\infty L_n$. 
We denote by $\pi$ the representation of $\fA(L)$ on $H_A$ determined by 
$$\pi(a(f))=\pi_A(a(S_{a_{2n}}f))=\alpha^\Phi_{a_{2n}}(\pi_A(a(f))),\quad f\in L_n.$$
Since $\pi(\fA(L))=\pi_A(\fA(K_O))$, it suffices to show that $\pi$ is a type I representation. 
Let $V$, $V_n$, and $W_n$ be the normalized vectors obtained in Lemma \ref{decomposable vector}. 
We set $\varphi(X)=\inpr{\pi(X)V\Omega_A}{V\Omega_A}$ for $X\in \fA(L)$. 
Then $\varphi$ is a state of $\fA(L)$ whose GNS representation is quasi-equivalent to $\pi$. 
We show that $\varphi$ is pure using Lemma \ref{product state}. 

Lemma \ref{making a state} shows that there exists a $\gamma$-invariant pure state $\varphi_n$ of $\fA(L_n)$ 
satisfying $V_n^*\pi_A(X)V_n=\varphi_n(X)1$ for $\forall X\in \fA(L_n)$. 
Let $X_i\in \fA(L_i)$, $i=1,2,\cdots,n$. 
Then 
\begin{eqnarray*}\lefteqn{V^*\pi_A(X_0X_1\cdots X_n)V}\\
&=&W_{2n}^*V_{2n}^*\cdots V_1^*V_0^*\pi_A(X_0)\alpha^\Phi_{a_2}(\pi_A(X_1))\cdots \alpha^\Phi_{a_{2n}}(\pi_A(X_n))V_0V_1\cdots V_{2n}W_{2n}\\
 &=&W_{2n}^*V_{2n}^*\cdots V_1^*V_0^*\pi_A(X_0)V_0V_1\pi_A(X_1)\cdots \alpha^\Phi_{a_{2n}-a_2}(\pi_A(X_n))V_2\cdots V_{2n}W_{2n} \\
 &=&\varphi_0(X_0)W_{2n}^*V_{2n}^*\cdots V_2^*\pi_A(X_1)\cdots \alpha^\Phi_{a_{2n}-a_2}(\pi_A(X_n))V_2\cdots V_{2n}W_{2n} \\
 &=&\varphi_0(X_0)\varphi_1(X_1)W_{2n}^*V_{2n}^*\cdots V_4^*\pi_A(X_2)\cdots \alpha^\Phi_{a_{2n}-a_4}(\pi_A(X_n))V_4\cdots V_{2n}W_{2n}\\
 &=&\cdots=\varphi_0(X_0)\varphi_1(X_1)\cdots \varphi_n(X_n). 
\end{eqnarray*}
Thus Lemma \ref{product state} shows that $\varphi$ is a pure state, and in consequence, $\pi$ is a type I representation. 
\end{proof}

In order to apply Theorem \ref{CABATIF} to concrete examples, we state the assumptions of Theorem \ref{CABATIF} 
in terms of the regular part $\hat{\Phi}_0$ of the Fourier transform $\hat{\Phi}$. 

\begin{lemma}\label{integral} Let the notation be as in Theorem \ref{CABATIF}. 
Assume that $\Phi$ is an even function. 
Then
\begin{itemize}
\item [(i)] The assumption of Theorem \ref{CABATIF},(i) holds if and only if 
$$\int_0^\infty \sum_{n=0}^\infty\min\{x,|I_n|\}\tr(|\hat{\Phi}_0(x)|^2)dx<\infty.$$
\item [(ii)] The assumption of Theorem \ref{CABATIF},(ii) holds if and only if 
$$\int_0^\infty  |O\ominus (O+x)|\tr(|\hat{\Phi}_0(x)|^2)dx<\infty,$$
where $O\ominus (O+x)$ is the symmetric difference of $O$ and $O$ translated by $x$. 
\end{itemize}
\end{lemma}

\begin{proof} (i) The statement follows from $|(I_n\setminus (I_n+t))|=\min\{|t|,|I_n|\}$ and  Lemma \ref{interval},(ii).

(ii) We set $J_{-1}:=(-\infty,0)$, $J_0:=(a,\infty)$, and $J_n=I_{2n-1}$ for $n\in \N$. 
Then 
$$\|(1_{\tK}-P_O)C_\Phi P_O\|_\HS^2=\sum_{m=-1}^\infty\sum_{n=0}^\infty\|P_{J_m}C_\Phi P_{I_{2n}}\|_{H.S.}^2.$$
The statement follows from this and Lemma \ref{interval},(ii). 
\end{proof}

Lemma \ref{PQ} implies $\|(1_{\tK}-P_E)C_\Phi P_E\|_\HS^2=\|C_{1-\Phi}P_EC_\Phi\|_\HS^2$. 
Thus by using Fourier transform, we can also get the following criteria, though we do not use them in this paper.

\begin{lemma} Let the notation be as in Theorem \ref{CABATIF}.
Then 
\begin{itemize}
\item [(i)] The assumption of Theorem \ref{CABATIF},(i) holds if and only if 
$$\int_{\R^2}\frac{\tr(|\Phi(p)-\Phi(q)|^2)}{|p-q|^2} \sum_{n=0}^\infty\sin^2\frac{|I_n|(p-q)}{2}dpdq<\infty.$$
\item [(ii)] The assumption of Theorem \ref{CABATIF},(ii) holds if and only if 
$$\int_{\R^2}\tr(|\Phi(p)-\Phi(q)|^2)|\hat{\chi_O}(p-q)|^2dpdq<\infty.$$
\end{itemize}
\end{lemma}
\section{Examples}\label{examples}

Applying Theorem \ref{CABATIF} to concrete sequences, we get the following theorem, which provides us with 
a computable invariant for type III Toeplitz CAR flows. 

\begin{theorem}\label{Sobolev 2} Let $\Phi\in L^\infty(\R)\otimes M_N(\C)$ be an admissible symbol satisfying 
$\Phi(p)=\Phi(-p)$ for all $p\in \R$, and let $0<\mu<1$. 
We set $a_0=0$, 
$$a_n=\sum_{k=1}^n\frac{1}{k^{1/(1-\mu)}},\quad n\in \N,$$ 
and $a=\lim_{n\to \infty}a_n$.    
Then the following three conditions are equivalent: 
\begin{itemize} 
\item [$(1)$] The type I factorization $\{\cA^{\Phi}_a(a_n,a_{n+1})\}_{n=0}^\infty$ is a CABATIF.\medskip 
\item [$(2)$] 
$$\int_0^\infty\int_0^\infty\frac{\tr(|\Phi(p)-\Phi(q)|^2)}{|p-q|^{1+\mu}}dpdq<\infty.$$
\item [$(3)$]
$$\int_0^\infty x^\mu \tr(|\hat{\Phi}_0(x)|^2)dx<\infty.$$
\end{itemize}
Moreover, 
\begin{itemize}
\item [(i)] If $\{\cA^{\Phi}_a(a_n,a_{n+1})\}_{n=0}^\infty$ is a CABATIF, then 
$$\int_0^\infty\tr(|\Phi(2p)-\Phi(p)|^2) \frac{dp}{p^{\mu}}<\infty.$$
\item [(ii)] If $\Phi$ is differentiable and 
$$\int_0^\infty \tr(|\Phi'(p)|^2)p^{2-\mu}dp<\infty,$$
then $\{\cA^{\Phi}_a(a_n,a_{n+1})\}_{n=0}^\infty$ is a CABATIF. 
\end{itemize}
\end{theorem}

\begin{proof} The statement follows from Lemma \ref{Besov}, Lemma \ref{integral}, 
and Lemma \ref{Oestimate} below applied to $h(x)=x^{\mu-1}$. 
\end{proof}

The following lemma is more or less \cite[Lemma 8.6]{IS}. 

\begin{lemma}\label{Oestimate} Let $h(x)$ be a non-negative strictly decreasing continuous function on 
$(0,\infty)$ satisfying 
$\lim_{x\to +0}h(x)=\infty$, $\lim_{x\to \infty}h(x)=0$, and 
$$\int_0^1h(x)dx<\infty.$$ 
We set $a_0=0$, 
$$a_n=\sum_{k=1}^nh^{-1}(k),\quad n\in \N,$$
$I_n=(a_n,a_{n+1})$, and $O=\bigcup_{n=0}^\infty I_{2n}$. 
Then the sequence $\{a_n\}_{n=0}^\infty$ converges, and  
$$x(h(x)-1)\leq |O\ominus (O+x)|\leq 2\sum_{n=0}^\infty \min\{x,|I_n|\}\leq 2\int_0^x h(t)dt,\quad \forall x>0.$$
\end{lemma}

\begin{proof}
Note that we have 
$$\sum_{k=n+1}^\infty h^{-1}(k)\leq \int_0^{h^{-1}(n+1)}h(t)dt-nh^{-1}(n+1),$$
and in particular, the sequence $\{a_n\}_{n=0}^\infty$ converges. 
Since $\min\{x,|I_n|\}=|I_n\setminus (I_n\pm x)|$, the middle inequality follows from the definition of $O$. 

For fixed $x>0$, we take the unique non-negative integer $n$ satisfying $h^{-1}(n+1)<x\leq h^{-1}(n)$ (or equivalently, $n\leq h(x)<n+1$). 
Then 
\begin{eqnarray*}
\sum_{k=0}^\infty \min\{x,|I_k|\}&=&\sum_{k=0}^{n-1} x+\sum_{k=n}^\infty |I_k|=nx+\sum_{k=n}^\infty h^{-1}(k+1)\\
&\leq&\int_0^{h^{-1}(n+1)}h(t)dt+n(x-h^{-1}(n+1))\\
&\leq&\int_0^xh(t)dt. 
\end{eqnarray*}
When $n$ is even, counting only contribution from $\{I_{2k}\}_{k=0}^{(n-2)/2}$, we get 
$$|(U+x)\setminus U|\geq \frac{n}{2}x.$$ 
In a similar way, we get $$|U\setminus (U+x)|\geq \frac{n}{2}x,$$ and so
$$|U\ominus (U+x)|\geq nx\geq (h(x)-1)x.$$ 
When $n$ is odd, we have $|(U+x)\setminus U|\geq \frac{n+1}{2}x$ and $|U\setminus (U+x)|\geq \frac{n+1}{2}x$ 
in a similar way, which shows  $|U\ominus (U+x)|\geq xh(x)$. 
\end{proof}

Now we apply Theorem \ref{Sobolev 2} to concrete examples. 

\begin{theorem}\label{uncountable} For $\nu>0$, let $\theta_\nu(p)=(1+p^2)^{-\nu}$, and let 
$$\Phi_\nu(p)=\frac{1}{2}\left(
\begin{array}{cc}
1 &e^{i\theta_\nu(p)}  \\
e^{-i\theta_\nu(p)} &1 
\end{array}
\right). 
$$
Then $\Phi_\nu$ is admissible. 
Let $\alpha^\nu:=\alpha^{\Phi_\nu}$ be the corresponding Toeplitz CAR flow. 
\begin{itemize}
\item [(i)] If $\nu>1/4$, then $\alpha^\nu$ is of type I$_2$. 
\item [(ii)] If $0<\nu\leq 1/4$, then $\alpha^\nu$ is of type III. 
\item [(iii)] If $0<\nu_1<\nu_2\leq 1/4$, then $\alpha^{\nu_1}$ and $\alpha^{\nu_2}$ are not cocycle conjugate. 
\end{itemize}
\end{theorem}

\begin{proof}
The fact that $\Phi_\nu$ is admissible follows from Theorem \ref{Sobolev},(ii). 
(i) and (ii) follow from Theorem \ref{dichotomy}. 
To show (iii), we choose $\mu$ in the interval $(1-4\nu_2,1-4\nu_1)$, which satisfies $0<\mu<1$. 
Applying Theorem \ref{Sobolev 2},(i),(ii) to this $\mu$ and $\Phi=\Phi_{\nu_i}$, $i=1,2$, we see that 
$\{\cA^{\Phi_{\nu_2}}_a(a_n,a_{n+1})\}_{n=0}^\infty$ is a CABATIF, 
while $\{\cA^{\Phi_{\nu_1}}_a(a_n,a_{n+1})\}_{n=0}^\infty$ is not. 
Therefor $\alpha^{\nu_1}$ and $\alpha^{\nu_2}$ are not cocycle conjugate. 
\end{proof}

\begin{remark} Let $\Phi$ be as in Example \ref{nonconverging}, and let $\mu$ and $\{a_n\}_{n=0}^\infty$ be as 
in Theorem \ref{Sobolev 2}. 
Then Theorem \ref{Sobolev 2},(i) implies that $\{\cA^\Phi_a(a_n,a_{n+1})\}_{n=0}^\infty$ is not a 
CABATIF for any $0<\mu<1$. 
This shows that $\alpha^\Phi$ is not cocycle conjugate to $\alpha^\nu$ for any $\nu$. 
\end{remark}



\thebibliography{99}
\bibitem{AW} H. Araki and E. J. Woods, 
\textit{Complete Boolean algebras of type I factors.} 
Publ. Res. Inst. Math. Sci. Ser. A \textbf{2} (1966), 157-242.
 
\bibitem{Ara} H. Araki,  
\textit{On quasifree states of CAR and Bogoliubov automorphisms.}  
Publ. Res. Inst. Math. Sci. \textbf{6}  (1970/71), 385--442.

\bibitem{Ar} W. Arveson, 
\textit{Continuous analogues of Fock spaces IV:  Essential states.} 
Acta Math. \textbf{164} (3/4) 265-300, 1990. 
  
\bibitem{Arv} W. Arveson, 
\textit{Non-commutative Dynamics and $E$-semigroups.} 
Springer Monograph in Math. (Springer 2003).



\bibitem{pdct} B. V. Rajarama Bhat and R. Srinivasan, 
\textit{On product systems arising from sum systems.} 
Infinite dimensional analysis and related topics, Vol. 8, Number 1, March 2005.


\bibitem{F} J. J. Foit, 
\textit{Abstract twisted duality for quantum free Fermi fields.} 
Publ. Res. Inst. Math. Sci. \textbf{19} (1983), 729--741. 


\bibitem{I} M. Izumi, 
\textit{A perturbation problem for the shift semigroup.} 
J. Funct. Anal.  \textbf{251}  (2007), 498--545.

\bibitem{I1} M. Izumi, 
\textit{Every sum system is divisible.}
Trans. Amer. Math. Soc. \textbf{361} (2009), 4247--4267. 

\bibitem{IS} M. Izumi and R. Srinivasan, 
\textit{Generalized CCR flows.}  
Comm. Math. Phys. \textbf{281}  (2008),  529--571.


\bibitem{L} V. Liebscher, 
\textit{Random sets and invariants for (type II) continuous tensor product systems of Hilbert spaces.} 
to appear in Mem. Amer. Math. Soc. arXiv:math/0306365. 



\bibitem{PO0} R. T. Powers, 
\textit{An index theory for semigroups of $*$-endomorphisms of $B(H)$ and type $II_1$ factors.} 
Can. J. Math., \textbf{40} (1988), 86-114.

\bibitem{Po1} R. T. Powers, 
\textit{A nonspatial continuous semigroup of $*$-endomorphisms of $B(H)$.} 
Publ. Res. Inst. Math. Sci.  \textbf{23} (1987), 1053-1069.

\bibitem{PS} R. T. Powers, and E. St{\o}rmer, 
\textit{Free States of the Canonical Anticommutation Relations.} 
Comm. Math. Phys., \textbf{16} (1970), 1-33.

\bibitem{Pr} G. L. Price, B. M. Baker, P. E. T. Jorgensen and P. S. Muhly, 
(Editors), {\it Advances in Quantum Dynamics}. 
(South Hadley, MA, 2002) Contemp.
Math. 335, Amer. Math. Society, Providence, RI (2003).


\bibitem {Ta} M. Takesaki,  
\textit{Theory of Operator Algebras. I.} 
Encyclopaedia of Mathematical Sciences, 124. 
Operator Algebras and Non-commutative Geometry, 5. 
Springer-Verlag, Berlin, 2002. 

\bibitem{T1} B. Tsirelson, 
\textit{Non-isomorphic product systems.} 
Advances in Quantum Dynamics (South Hadley, MA, 2002), 273--328, 
Contemp. Math., \textbf{335}, Amer. Math. Soc., Providence, RI, 2003. 

\bibitem{T2} B. Tsirelson,  
\textit{Spectral densities describing off-white noises.} 
Ann. Inst. H. Poincar  Probab. Statist. \textbf{38} 
(2002), 1059--1069.
\end{document}